\documentclass[11pt,reqno]{amsart} 

\numberwithin{equation}{section} 

\usepackage{amsmath,amsthm}           
\usepackage{amssymb}                  
\usepackage{mathtools}                
\mathtoolsset{showonlyrefs}          

\usepackage{a4wide}                   

\usepackage{enumitem}                

\usepackage{stmaryrd}                
\usepackage{xfrac}                   
\usepackage{faktor}                  
\usepackage{physics}                 

\usepackage[colorlinks]{hyperref}   

\usepackage{subcaption}             

\usepackage{tabstackengine}         
\setstackEOL{;}                     
\setstackTAB{,}                     
\setstacktabbedgap{1ex}             
\setstackgap{L}{1.0\normalbaselineskip} 

\usepackage{tikz}                    
\usepackage{tikz-cd}                
\usepackage{pgfplots}               
\pgfplotsset{compat=1.18}           

\usetikzlibrary{arrows.meta, matrix, backgrounds, patterns, fadings}
\usetikzlibrary{arrows, decorations.pathmorphing, calc}
\usetikzlibrary{shapes, automata, petri, positioning}

\tikzset{
	place/.style={circle, thick, draw=black, fill=gray!50, minimum size=20mm},
	state/.style={circle, thick, draw=blue!75, fill=blue!20, minimum size=20mm},
	cross/.pic = {
		\draw[rotate = 45] (-0.2,0) -- (0.2,0);
		\draw[rotate = 45] (0,-0.2) -- (0, 0.2);
	}
}

\newcommand{\inner}[1]{\langle #1\rangle}
\newcommand{\floor}[1]{\lfloor #1 \rfloor}
\newcommand{\ceil}[1]{\lceil #1 \rceil}

\newcommand{\eps}{\epsilon}
\newcommand{\rmd}{\mathrm{d}}   
\newcommand{\ndef}[1]{{\textcolor{blue}{#1}}}

\renewcommand{\r}{\mathbf{r}}

\newcommand{\T}{\mathbb{T}}
\newcommand{\where}{\quad\text{where}\quad}

\theoremstyle{plain}
\newtheorem{Th}{Theorem}[section]
\newtheorem{Lemma}[Th]{Lemma}

\newtheorem{Prop}[Th]{Proposition}

\theoremstyle{definition}
\newtheorem{Def}[Th]{Definition}

\newtheorem{Rem}[Th]{Remark}
\newtheorem{?}[Th]{Problem}

\newtheorem{Ass}[Th]{Assumption}

\title[]
{Continuously Parametrised Porous Media Model and Scaling Limits of Kinetically Constrained Models}

\author{Gabriel S. Nahum}
\address{Gabriel S. Nahum\\
	inria Lyon, DRACULA team\\
	56 Bd. Niels Bohr, 69199 Villeurbanne, France}
\email{{\tt gabriel.nahum@protonmail.com}}
\usepackage{multicol}
\begin{document}

\begin{abstract}
	We investigate the emergence of non-linear diffusivity in kinetically constrained, one-dimensional symmetric exclusion processes satisfying the gradient condition. Recent developments introduced new gradient dynamics based on the Bernstein polynomial basis, enabling richer diffusive behaviours but requiring adaptations of existing techniques. In this work, we exploit these models to generalise the Porous Media Model to non-integer parameters and establish simple conditions on general kinetic constraints under which the empirical measure of a perturbed version of the process converges. This provides a robust framework for modelling non-linear diffusion from kinetically constrained systems.
\end{abstract}

\maketitle

\renewcommand{\thefootnote}{} 
\footnote{$^\dag$\textsc{
    inria lyon DRACULA,
    56 Bd. Niels Bohr, 69199 Villeurbanne, France
    }
    \par\nopagebreak
    \;\;
    \textit{E-mail address}: 
    \texttt{gabriel.nahum@protonmail.com}}
\renewcommand{\thefootnote}{\arabic{footnote}}

\tableofcontents

\section{Introduction}

This work lies within the framework of statistical mechanics, specifically in the study of interacting particle systems, consisting of a large number of agents interacing with each other, evolving according to a Markovian law. A central question in this context is to determine which types of macroscopic diffusivities can emerge from such microscopic, stochastic models. We focus on one of the simplest classes of these systems: one-dimensional, symmetric exclusion processes. Particles are randomly distributed on the lattice, subject to the exclusion rule that at most one particle may occupy each site. The system evolves through nearest-neighbour jumps, without directional bias, and with rates depending on the local configuration. These jumps are governed by independent Poisson processes, each associated with a pair of sites and with rates given determined by the local state.

A foundational toy model was introduced in \cite{GLT} to capture non-linear diffusion of the form $\rho^n$ for positive integer values of $n$. In this dynamics, a pair of neighbouring sites exchange their occupation values at a rate determined by the number of boxes of $n$ aligned particles surrounding the pair. If the surrounding configuration contains sufficient vacancies, the exchange is suppressed. This model, commonly referred to as the Porous Media Model (PMM), enjoys technical properties that make it analytically tractable. It has since been studied in various contexts, such as systems with periodic boundary conditions and in higher dimensions \cite{GLT}; open systems, which lead to different boundary conditions depending on the scaling of the boundary dynamics \cite{BDGN,renatoenergy}; extensions to long-range dynamics \cite{CG24}; and superpositions of different PMMs \cite{s2f,CGN24}. Our work lies within the context of the last two, that we now describe.

In \cite{s2f}, the PMM was extended to include non-integer parameters, defining a dynamics associated with $\rho^m$ for $m \in (0,2)$. As non-integer exponents cannot be obtained through a direct combinatorial extension of the PMM, the authors applied the generalised binomial theorem to expand the diffusivity, and then considered weighted superpositions of PMMs with different parameters. However, the process becomes ill-defined for $m > 2$ due to the emergence of negative rates.

More recently, in \cite{CGN24}, the ideas of \cite{s2f} were extended to construct diffusivities of the form $\sum_{k\geq 0} b_k \rho^k$, though still falling short of a complete generalisation of the PMM for $m > 2$. Nonetheless, that work provided precise conditions on the coefficients $b_k$ under which a statistical analysis is possible. In parallel, \cite{N24} established that PMMs cannot produce diffusivities with multiple degenerate roots of order greater than two, and introduced a new collection of processes, referred to as Bernstein models, to address this limitation. This class is associated with the Bernstein polynomial basis, and generalises the PMM into a two-parameter model exhibiting degenerate diffusion at the extreme values $\rho = 0$ and $\rho = 1$. In the present work, we exploit this class to construct a generalised PMM that continuously interpolates between models associated with an integer parameter.

Because the work in \cite{CGN24} is based on the PMM, producing diffusivities expressed in the monomial basis $\{\rho^n\}_{n\geq 0}$, the results cannot be directly applied to the processes introduced in \cite{N24}, and consequently not to the generalised PMM introduced here, which are naturally expressed in a different basis. 

It was shown in \cite{N24} that the Bernstein models are associated with a new collection of gradient dynamics, linked to the PMM and associated with the monomial basis $\{\rho^n\}_{n\geq 0}$, and with a combinatorial structure similar enough for arguments analogous to the ones in \cite{CGN24} to apply. However, rewriting a series originally expressed in the Bernstein basis into the standard basis leads to excessively large coefficients $(|b_k|)_{k\geq 0}$, which severely restricts the class of attainable diffusivities.

This observation highlights the need for an analysis independent of the choice of basis. This is precisely what we pursue in the second part of this work. Our goal is to provide simple conditions under which the hydrodynamic limit of the empirical measure holds. This corresponds to a (weak) law of large numbers for the local average of particles, showing that, under diffusive rescaling, it converges to the solution of a deterministic partial differential equation.

It is important to note that we study perturbed versions of kinetically constrained models. For non-irreducible dynamics, determining whether the system mixes sufficiently to allow a statistical description is a deep and difficult problem—still open for the PMM under general initial data \cite{GLT}—and we do not address it here. Instead, we perturb the dynamics with the Symmetric Simple Exclusion Process at a strength that is macroscopically negligible. This guarantees sufficient local mixing without modifying the hydrodynamic equation.

We show that, for jump rates depending on the system size, the hydrodynamic limit holds under simple conditions: a uniform bound on the maximum rate; a bound on the sum of rates over local configurations within a fixed-size box; a uniform Cauchy condition on a function associated with the so-called gradient condtition, also also its uniformly boundedness from above. These conditions are satisfied by the generalised PMM and provide a robust framework for modelling a broader class of diffusive behaviours.

\subsection{Outline of the paper}
In Section \ref{sec:context}, we define the relevant processes and state the main result. Specifically, Subsection \ref{subsec:model} is devoted to introducing the generalised PMM and other properties relevant to this work, while Subsection \ref{subsec:kcm} presents some general conditions for Kinetically Constrained Models under which we can guarantee the hydrodynamic limit for the empirical measure associated with a perturbed process. In Subsection \ref{subsec:topo}, we present the topological setting and the notions of weak solutions; and in Subsection \ref{subsec:hl}, we present the main result of this work: the hydrodynamic limit for the empirical measure associated with the perturbed process. We conclude with some technical comments.

The convergence towards the hydrodynamic limit is shown in Section \ref{sec:hl}. At the core of the proof are the so-called \textit{replacement lemmas}, presented in Section \ref{sec:rep-lem}. In Section \ref{sec:energy}, we show an energy estimate, guaranteeing uniqueness of solutions of the hydrodynamic equation.

\section{Setup and Main result}\label{sec:context}
For readability, any new notation introduced within a paragraph and used throughout the text will be highlighted in \ndef{blue}.

Let $\ndef{\mathbb{N}_+}$ be the set of positive natural numbers and denote by $\ndef{N} \in \mathbb{N}_+$ a scaling parameter. Our discrete lattice is $\ndef{\mathbb{T}_N}$, the one-dimensional discrete torus, $\mathbb{T}_N = \{1, \dots, N\}$ with the identification $0 \equiv N$. For any $x < y \in \mathbb{Z}$, that can be viewed as elements in $\mathbb{T}_N$ by considering their standard projections, we define the discrete interval $\ndef{\llbracket x,y\rrbracket}$, composed of all the points between $x, y$ (including $x, y$) in $\mathbb{T}_N$, where the order has been inherited from the one in $\mathbb{Z}$.

The central dynamics of this work is an interacting particle system, following a Markovian law, satisfying the exclusion rule and situated on the discrete torus $\mathbb{T}_N$. A configuration of particles is an element of the state space $\ndef{\Omega_N} = \{0,1\}^{\mathbb{T}_N}$ and will be recurrently denoted by the letters $\eta$ and $\xi$. We denote by $\ndef{\eta(x)} \in \{0,1\}$ the occupation value of $\eta \in \Omega_N$ at the site $x \in \mathbb{T}_N$.

The two main contributions of this work are the introduction of a collection of models continuously parametrised, interpolating the Porous Media Models with integer parameters; and the proof of the hydrodynamic limit for a class of symmetric exclusion processes with much more relaxed constraints compared to the current state of the art. Our proof includes the aforementioned collection of models, as well as superpositions of the models introduced in \cite{N24}. We begin by presenting the novel processes, and then present the context of our proof and how it encompasses this collection.

\subsection{Context and the generalized Porous Media Model}\label{subsec:model}
We start by introducing the operators associated with the processes that we are going to present, and fix some notation.  
\begin{Def}
In what follows, let $\eta\in\Omega_N$ and $x,y,z\in\mathbb{T}_N$ be arbitrary:
\begin{itemize}
        \item Let $\pi_x:\Omega_N\to\{0,1\}$ be the projection $\pi_x(\eta)=\eta(x)$;
        
        \item $\tau:\eta\mapsto\tau\eta$ is defined as the shift operator, $\pi_x(\tau\eta)=\pi_{x+1}(\eta)$, and short-write $\tau_i=\circ_{j=1}^i\tau$ for its $i$-th composition;

        \item Denote by $\theta_{x,y}$ the operator that exchanges the occupation-value of the sites $x,y$,
        \begin{align}
            (\theta_{x,y}\eta)(z)
            =\eta(z)1_{z\neq x,y}
            +\eta(y)1_{z=x}
            +\eta(x)1_{z=y};
        \end{align}

        \item Define the map $\eta\mapsto\overline{\eta}$ through $\overline{\eta}(x)=1-\eta(x)$;
        
    \end{itemize}
For $f:\Omega_N\to\mathbb{R}$ arbitrary,     
    \begin{itemize}
    	\item For $\mathcal{O}=\tau,\theta_{x,y}$, let $\mathcal{O} f(\eta):=f(\mathcal{O}\eta)$;
    	
    	\item Introduce the linear operators $\nabla_{x,y},\nabla,\Delta$ through 
    	\begin{align}
    		\nabla_{x,y}[f]:=\theta_{x,y}f-f
    		,\quad
    		\nabla[f]:=\tau f-f 
    		\quad\text{and}\quad
    		\Delta[f]=\nabla^2[\tau_{-1}f]
    		;
    	\end{align}
    	
    	\item Shorten $|f|_{\infty}=\sup_{\eta\in\Omega_N}|f(\eta)|$.
    \end{itemize}
    
Note that $\nabla$ corresponds to the forward difference operator and $\Delta$ to the discrete Laplacian operator.
\end{Def}

This work is situated in the context of symmetric exclusion-type dynamics in a homogeneous environment. Specifically, these processes are continuous-time Markov chains, characterised by their infinitesimal generator, $\mathcal{L}$, and defined through its action on $f:\Omega_N\to\mathbb{R}$ by
\begin{align}
	\mathcal{L}_N[f]
	&=\sum_{x\in\T_N}
	\left(\tfrac12\mathbf{r}_{x,x+1} +
	\tfrac12\mathbf{r}_{x,x+1}\right)
	\nabla_{x,x+1}[f]
	,
\end{align}
where, for each $\eta\in\Omega_N$, the rate for a jump from $x$ to $x+1$ is given by $\mathbf{r}_{x,x+1}(\eta)$, while the rate from $x+1$ to $x$ is given by $\mathbf{r}_{x+1,x}(\eta)$. In our specific context, the rates are expressed as $\mathbf{r}_{x,x+1}=\mathbf{c}_{x,x+1}\mathbf{e}_{x,x+1}$ and $\mathbf{r}_{x+1,x}=\mathbf{c}_{x+1,x}\mathbf{e}_{x+1,x}$, where $\mathbf{c}_{x,x+1}\geq 0$ is the \textit{kinetic} constraint, and $\mathbf{e}_{x,x+1}(\eta)=\eta(x)(1-\eta(x+1))$ is the \textit{exclusion} constraint for the hopping from site $x$ to site $x+1$; $\mathbf{c}_{x+1,x}$ and $\mathbf{e}_{x+1,x}$ are the analogous constraints for a jump from $x+1$ to $x$. Both $\mathbf{c}_{x,x+1}$ and $\mathbf{c}_{x+1,x}$ depend on the range of interaction, possibly taking different values for different local configurations. The symmetry of the process is considered at the level of the constraints, $\mathbf{c}_{x,x+1}=\mathbf{c}_{x+1,x}$, and the homogeneity of the media is expressed through $\mathbf{c}_{x,x+1}=\tau_x\mathbf{c}_{0,1}$. Conveniently, the symmetry implies that the generator is reversible with respect to the Bernoulli product measure $\nu_\alpha^N:=\otimes_{x\in\T_N}\text{Ber}(\alpha)$, with $\alpha\in[0,1]$ and $\text{Ber}(\alpha)$ denoting the Bernoulli distribution. This allows us to express
\begin{align}\label{generic:gen}  
	\mathcal{L}_N[f](\eta)  
	=  
	\sum_{x\in\mathbb{T}_N}  
	\mathbf{c}(\tau_x\eta)  
	\nabla_{x,x+1}[f](\eta),  
\end{align}  
associating each symmetric process, as just described, with a constraint $\mathbf{c}$.

We are now ready to recall the Porous Media Model \cite{GLT}.
\begin{Def}\label{def:pmm}
	Fixed $n\in\mathbb{N}_+$, the PMM($n$) is the process given by the generator $\mathcal{L}_N^n$, defined by replacing in \eqref{generic:gen} the constraint $\mathbf{c}$ by 
    \begin{align}\label{pmm-cons}
		\mathbf{p}_{n}(\eta):=\sum_{i=0}^n
		\mathbf{p}_{n}^j
        ,
        \quad\text{with}\quad
        \mathbf{p}_{n}^j
        =\prod_{\substack{i=-(n+1)+j\\i\neq 0,1}}^j
		\eta(i)
		.
	\end{align}
\end{Def}

The PMM is a toy model for degenerate diffusion, exhibiting key properties such as the existence of mobile clusters and blocked configurations, as we shall explain more thoroughly later on. This process, along with several of its extensions, has been actively studied in the literature.

In \cite{GLT}, it was shown that the (macroscopic) local particle density is governed by the Porous Medium Equation (PME), $\partial_t \rho = \partial_u (\rho^n \partial_u \rho)$. There, this was formalised via the relative entropy method, assuming initial data bounded away from both $0$ and $1$; and via the entropy method by perturbing the dynamics with the Symmetric Simple Exclusion Process evolving on a subdiffusive time scale. This perturbation ensured the irreducibility of the resulting process, allowing the derivation of the PME for general initial data. All results were obtained in one dimension with periodic boundary conditions.

In \cite{BDGN}, the authors studied the perturbed PMM in a one-dimensional open system, where creation and annihilation of particles occur at the boundaries. Depending on the scaling of the injection and removal rates, Dirichlet, Robin, or Neumann boundary conditions were derived.

The present work is an extension of \cite{s2f}. In that work, the authors first addressed the problem of generalising the dynamics in order to derive the PMM for non-integer values of $n$. Since the process is of exclusion type, this could not be achieved by a purely combinatorial generalisation of the dynamics. Instead, the authors proceeded analytically, viewing the collection $\{\mathbf{p}_{n}\}_{n \geq 0}$ as a monomial basis, and, via the generalised binomial theorem, defined the constraint for a fixed $m \in \mathbb{R}$ as follows:
\begin{align}\label{c:s2f}
	\mathbf{c}(\eta)
    \equiv
    \sum_{k=0}^{\ell_N} \binom{m-1}{k} (-1)^{k} \mathbf{p}_{k}(\overline\eta),
\end{align}
where $2\leq \ell_N\rightarrow+\infty$ as $N\to+\infty$, and we recall that the generalized binomial coefficient is given by the formula
\begin{align}\label{eq:binom}
	\binom{c}{k} = \frac{(c)_k}{k!} = \frac{c(c-1)\cdots(c-(k-1))}{k!},\qquad c\in\mathbb{R}
	.
\end{align}
It turns out that the constraint is well-defined for values of $m \in (0,2]$, and for such values the corresponding hydrodynamic equation is given by $\partial_t \rho = \partial_u^2 \rho^m$, thereby encompassing the Porous Medium Equation (PME) for $m \in (1,2]$, and also the Fast Diffusion Equation in the regime $m \in (0,1)$. It is also worth mentioning that the process coincides with PMM($2$) as $m \nearrow 2$, and with the SSEP as $m \searrow 1$.

The reason the model becomes ill-defined for $m > 2$ is related to the lack of monotonicity in the sequence $(\mathbf{p}_n - \mathbf{p}^{n+1})_{n \geq 0}$, which leads to $\mathbf{c}(\eta) < 0$ for certain configurations $\eta$. In a similar spirit, other diffusivities were derived by considering superpositions of PMMs with different weights satisfying specific growth conditions \cite{CGN24}, though these do not include the diffusion coefficient $D(\rho) = m \rho^{m-1}$ for $m > 2$.

The solution we present in this work is motivated by the collection of models introduced in a companion paper \cite{N24}, which are associated with the Bernstein polynomial basis, in contrast to the PMMs that are based on the standard monomial basis. Fixed any positive integer $n$, next constraint extends the diffusivity $\rho^n\to\rho^{n+m}$ for any $m\in[0,1]$.

\begin{Def}[Interpolating process]\label{def:pmm-gen}
Fix $n,\ell_N\in\mathbb{N}$ and $m\in[0,1]$ such that $2\leq \ell_N\leq N/2-2$, and $\ell_N\to+\infty$ as $N\to+\infty$. Define the generator $\mathcal{L}_N^m$ by replacing in \eqref{generic:gen} the constraint $\mathbf{c}$ by
\begin{align}\label{c}
    \mathbf{c}_{N}^{m}
    &=
    \mathbf{p}_{n}
    +
    \sum_{k=1}^{\ell_N}
    \binom{m}{k}(-1)^k
    \big(\mathbf{p}_{n}-
    \mathbf{p}_{n,k}
    \big)
\end{align}
with auxiliary constraints given by
\begin{align}
    \mathbf{p}_{n,k}
    =\frac{1}{n+k+1}
    \sum_{j=0}^{n+k}
    \mathbf{p}_{n,k}^{j}
\qquad\text{with}\qquad
    \mathbf{p}_{n,k}^{j}
    =
    \frac{\binom{\mathfrak{n}_j^{n+k}}{n}}{\binom{n+k}{n}}
    \mathbf{1}
    \big\{
    \mathfrak{n}_j^{n+k}
    \geq n+1
    \big\}
\end{align}
and where, for each $j$ as above, $\mathfrak{n}_j^{n+k}$ corresponds to the number of particles in the box $\llbracket-j,-j+n+k+1\rrbracket\backslash\{0,1\}$. In what follows, we also shorten 
\begin{align}
	\mathbf{c}_{N}^{m,j}
	=\mathbf{p}_{n}^j
	+
	\sum_{k=1}^{\ell_N}
	\binom{m}{k}(-1)^k
	\big(\mathbf{p}_{n}^j-
	\mathbf{p}_{n,k}^j
	\big).
\end{align}
\end{Def}

The constraint $\mathbf{c}_N^m$ is defined so as to interpolate between the integer Porous Media Models PMM($n$) and PMM($n+1$); it is related to the diffusion coefficient $D(\rho) = \rho^{m-1}$, and it satisfies the so-called gradient property. The latter two properties will be explained shortly. 

Regarding the interpolation, note that one can express
\begin{align}
	    \mathbf{c}_{N}^{m}
	&=\delta_N
	\mathbf{p}_{n}
	+
	m\mathbf{p}_{n,1}
	+
	\sum_{k=2}^{\ell_N}
	\abs{\binom{m}{k}}
	\mathbf{p}_{n,k}
	&&\text{with}\qquad
	\delta_N
	=
	1-\sum_{k=1}^{\ell_N}\abs{\binom{m}{k}},
\end{align}
which in particular shows that the model is well-defined. From $|\binom{m}{k}|\rightarrow0$ as $m\searrow0$ we obtain that $\lim_{m\searrow0}\mathbf{c}_N^m=\mathbf{p}_n$. Because $|\binom{m}{k}|:= 0$ for $m=1$ and any $k\geq 2$ and, remarkably, 
\begin{align}
	\mathbf{p}_{n,1}=\mathbf{p}_{n+1},
\end{align}
it holds that $\lim_{m\nearrow1}\mathbf{c}_N^m=\mathbf{p}_{n+1}$.

The Interpolating Process in the previous definition is classified as a non-cooperative, kinetically constrained, symmetric exclusion process. The non-cooperative aspect means that there exist \textit{mobile clusters} (m.c.)—groups of particles capable of performing excursions across the lattice and transporting mass throughout the system. 

To be more precise, let $\blacksquare$ denote a fixed finite box composed of particles and vacant sites. Representing a particle by $\bullet$ and a hole by $\circ$, we say that $\blacksquare$ constitutes a mobile cluster if it satisfies the following conditions:
	\begin{itemize}[label=$\triangleright$]
		\item \textit{Mobility}: the transitions
		\begin{align}\label{mobility}
			\blacksquare\bullet \leftrightarrow \bullet\blacksquare
			\quad\text{and}\quad
			\blacksquare\circ\leftrightarrow\circ\blacksquare
		\end{align}
		are possible with a finite number of jumps, and independent of the rest of the configuration; 
		\item \textit{Mass transport}: it is always possible for a jump to occur in a node, if there exists a cluster in the vicinity of the respective node, that is
		\begin{align}\label{mtransport}
			\blacksquare\circ\bullet\leftrightarrow\blacksquare\bullet\circ
			\quad\text{and}\quad
			\circ\bullet\blacksquare\leftrightarrow\bullet\circ\blacksquare
		\end{align}
		are always possible with a finite number of steps independently of the rest of the configuration.
	\end{itemize}
The mobile clusters are the collections of m.c.s of \textit{(i)} the PMM($n$), and \textit{(ii)} the processes induced by the constraints $\{\mathbf{p}_{n,k}\}_{1 \leq k \leq \ell_N}$. These correspond to collections of boxes of length at least $n+2$ and no larger than $\ell_N + 2$, containing at least $n+2$ particles and at least one vacant site. These boxes are, indeed, m.c.s because, within each such box, mixing is unconstrained: any pair $\bullet\circ$ (resp. $\circ\bullet$) entirely contained in one of the aforementioned boxes, $\blacksquare$, and located at a node, say $\{0,1\}$, can transition to $\circ\bullet$ (resp. $\bullet\circ$). This is due to the fact that removing $\bullet\circ$ (resp. $\circ\bullet$) from the configuration in $\blacksquare$ leaves a state with $\mathbf{c}_N^m > 0$. 

From this mixing property, the "mobility" and "mass transport" properties described in the previous display follow directly by an appropriate reorganisation of the cluster.
%

Another key property of the process is that it satisfies the gradient condition, as in \cite[Part II Subsection 2.4]{spohn:book}. This means that the algebraic current $\mathbf{j}_{x,x+1} := \tau_x \mathbf{j}_{0,1}$ associated with an arbitrary node $\{x,x+1\}$—defined via $\mathcal{L}_N[\pi_x](\eta) = -\nabla \mathbf{j}_{x,x+1}(\tau_{-1} \eta)$ and identified with $\mathbf{j}_{0,1}(\eta) = \mathbf{c}(\eta)\mathbf{e}_{0,1}(\eta) - \mathbf{c}(\eta)\mathbf{e}_{1,0}(\eta)$—satisfies the equation $\mathbf{j}_{0,1} = \nabla \mathbf{H}$ for some function $\mathbf{H}: \Omega_N \to \mathbb{R}$.

It is well known that this property greatly simplifies the analysis of the process (see \cite{PHD:clement} and references therein). In particular, the diffusion coefficient is given explicitly by $\mathbf{E}_{\nu_\alpha^N}[\mathbf{c}_N^m]$, in contrast to the general case where it is defined via a variational problem \cite[Part II Subsection 2.2]{spohn:book}.

To observe the gradient property, it is important to note that the generator $\mathcal{L}_N^m$ can be decomposed into the elementary processes introduced in \cite{N24}, which we now recall.

\begin{Def}
	 For each integers $0\leq n\leq L $, let $\mathcal{L}_N^{n,L}$ be the generator given by replacing in \eqref{generic:gen} the constraint $\mathbf{c}$ by $\mathbf{b}_{n,L}$, with 
	\begin{align}\label{rate:b}
		\mathbf{b}_{n,L}
		:=\frac{1}{L+1}\sum_{j=0}^{L}
		\mathbf{b}_{n,L}^j
		\quad\text{and}\quad
		\mathbf{b}_{n,L}^j(\eta)
		:=\sum_{j=0}^{L}\mathbf{1}\{\mathfrak{n}_j^{L}=n\}.
	\end{align}
For each $n,L$ we refer to the process associated with $\mathcal{L}_N^{n,L}$ as the Bernstein model, and specify its parameters by writing B($n$,$L$).	
\end{Def}

The Bernstein model was introduced as a toy-model for non-linear diffusivity, generalizing the PMM in the sense that B($n$,$n$)=PMM($n$) and defined in such a way that the underlying dynamics is of gradient type \cite[Proposition 2.7]{N24}. Moreover, it is identified the diffusivity
\begin{align}\label{binom}
	\text{E}_{\nu_\alpha^N}[\mathbf{b}_{n,L}]
	=\binom{L}{n}\alpha^n(1-\alpha)^{L-n}
	=:\text{B}_{n,L}(\rho)
	,
\end{align}
that corresponds to an element of the Bernstein polynomial basis. It is straightforward to verify that
\begin{align}\label{gen:decomp}
    \mathcal{L}_N^{m}
    =
    \delta_N
    \mathcal{L}_N^{n}
    +
    m\sum_{k=1}^{\ell_N}
    \abs{\binom{m}{k}}
    \sum_{\nu=n+1}^{n+k}
    \frac{\binom{\nu}{n}}{\binom{n+k}{k}}
    \mathcal{L}_N^{\nu,n+k},
\end{align}
and for this reason the PMM($m$) must also enjoy the gradient property. Regarding the diffusivity, we can see that 
$\text{E}_{\nu_\alpha^N}[\mathbf{c}_N^m]\xrightarrow{N\to\infty}m\alpha^{m-1}$ uniformly in $\alpha$, since from the inverse binomial transformation and the change of the sign of the generalized binomial coefficient,
\begin{align}
    \text{E}_{\nu_\alpha^N}
    [\mathbf{c}_N^m]
    &=
    m\rho^{n}
    \bigg(1-\sum_{k=1}^{\ell_N}
    \bigg|\binom{m}{k}\bigg|
    \bigg)
    +
    m\sum_{k=1}^{\ell_N}
    \abs{\binom{m}{k}}
    \bigg(
    \alpha^{n}
    -\frac{1}{\binom{n+k}{k}}
    \text{B}_{n,n+k}(\alpha)
    \bigg)
    \\
    &=m\alpha^{n}
    \sum_{k=0}^{\ell_N}
    (-1)^k\binom{m}{k}
    (1-\alpha)^k
    ,
\end{align}
and as $N\to\infty$, from the generalized binomial expansion, the summation in last line above converges uniformly to $\alpha^{m}$. In this way, one expects the local density of particles to be governed by the Porous Media Equation $\partial_t\rho=\partial_u(\rho^{n+m}\partial_u\rho)$.

\subsection{Scaling limit of exclusion KCMs}\label{subsec:kcm}
We are going to investigate the hydrodynamic equations for the local density of particles associated with two classes of models. We introduce some notation.
\begin{Def}
	Fix $N\in\mathbb{N}_+$ and let $\mathcal{L}_N$ be the generator as in \eqref{generic:gen}, associated with a constraint $\mathbf{c}_N:\Omega_N\to\mathbb{R}$, that can be generically expressed, for each $\eta\in\Omega_N$, as
\begin{align}
	\mathbf{c}_N(\eta)
	=\sum_{j=0}^{\ell_N}
	\mathbf{c}_N^j(\eta)
	\where
	\mathbf{c}_N^j(\eta)
	=
	w_{N,j}
	\prod_{w\in A_j}\eta(w)
\end{align}
for appropriate weights $\{w_{j},N\}_{j\geq 0}$ and sets $A_j\subset \T_N$ such that $0,1\notin A_j$, and some positive integer $\ell_N$. 

We introduce $\mathbf{r}_N^j:=(\mathbf{e}_{0,1}+\mathbf{e}_{1,0})\mathbf{c}_{N}^j$ for the rate function for an exchange in $\{0,1\}$ for the component $j$, and similarly introduce $\mathbf{r}_N=\sum_j\mathbf{r}_N^j$ for the contribuition of all the components. Moreover, we write $\mathcal{R}_{N,j}$ as the range of $\mathbf{r}_N^j$, that is, 
\begin{align}
	\mathcal{R}_{N,j}
	:=\bigcap_{A\in A_j^{\wedge}} A
	, &&
	A_j^{\wedge}=
	\{
	A\subset\T_N:\;
	A_j\cup\{0.1\}\subset A,\;A=\llbracket x,x+y\rrbracket\text{ for some } x,y\in\T_N
	\}
	,
\end{align}
and shorten $\underline{r}_{N,j}=\min\{\mathcal{R}_{N,j}\}$.
\end{Def}
\begin{Ass}\label{ass}
The constraints must satisfy the growth conditions
	\begin{align}\label{ass:rate}
		\lim_{N\rightarrow+\infty}
		&\frac1N|\mathbf{r}_N|_{\infty}
		=0,
		\\\label{ass:rate-total}
		\lim_{N\rightarrow+\infty}
		&\frac1N
		\bigg|
		\sum_{j=0}^{N}
		\sum_{z\in\mathcal{R}_{N,j}}
		\tau_z\mathbf{r}_{N}^j
		\bigg|_{\infty}
		<\infty
		.
	\end{align}
There exist some $\mathbf{H}_N$ such that $	\mathbf{c}_{N}(\eta)(\mathbf{e}_{0,1}(\eta)-\mathbf{e}_{1,0}(\eta))
=\nabla\mathbf{H}_{N}(\eta)$. 
Moreover, introducing
\begin{align}\label{ass:h-form}
	\mathbf{h}_{N}&:=\mathbf{H}_N-\mathbf{g}_{N},
	\\\text{with}\quad
	\mathbf{g}_{N}
	&:=-\sum_{j=0}^{N}
	\text{sign}(\underline{r}_{N,j})
	\sum_{i=\min\{0,\underline{r}_{N,j}\}}^{\max\{0,\underline{r}_{N,j}\}}
	\tau^i\mathbf{c}_N^j
	(\mathbf{e}_{i,i+1}-\mathbf{e}_{i+1,i})
	,
\end{align}
it must be that 
\begin{align}\label{ass:h-reg}
	&\lim_{L\to\infty}\lim_{N\to\infty}
	|
	\mathbf{h}_{N}
	-
	\mathbf{h}_{L}
	|_{\infty}=0
	,\\
	&\label{ass:h-bound}
	\exists\overline{h}>0\text{ independent of $N$ such that $|\mathbf{h}_{N}|_{\infty}<\overline{h}$.} 
\end{align}
\end{Ass}
Related to the previous assumptions is the function $\ndef{\Phi_L}:[0,1]\to\mathbb{R}_+$, for each $L\in\mathbb{N}_+$ fixed, given through
\begin{align}\label{phiL}
	\Phi_L(\alpha)
	:=\text{E}_{\nu_{\alpha}^{L^\star}}[\mathbf{h}_{L}]
	,
\end{align}
with $L^\star=|\cup_{j=0^{\ell_L}}\mathcal{R}_{L,j}|$ and any $\alpha\in[0,1]$. From the assumptions in \ref{ass:h-reg} and \ref{ass:h-bound}, the sequence $(\Phi_L)_{L\geq 0}$ converges uniformly to a function \ndef{$\Phi$}, that from the gradient condition we can infer to be non-decreasing and non-negative.

From the decomposition in \eqref{ass:h-form}, we see that $\text{E}_{\nu_{\alpha}^{N}}[\mathbf{h}_{N}] = \text{E}_{\nu_{\alpha}^{N}}[\mathbf{H}_{N}]$. Indeed, we will show that the term $\mathbf{g}_N$ has no macroscopic effect on the density level, even though the sequence $(|\mathbf{g}_N|_{\infty})_{N \geq 1}$ may diverge. We believe that $\mathbf{g}_N$ (modulo lattice translations) constitutes the main error in the approximation $\Phi_N \approx \Phi$, and that, in particular, $\mathbf{h}_N(\eta) \geq 0$.

This is supported by all the models introduced here, and is motivated by the strategy used to verify the gradient property in \cite{N24}, as well as by the approach to defining gradient dynamics in \cite[II Subsection 2.4, discussion just before (2.9)]{spohn:book}. The term $\mathbf{g}_N$ corresponds precisely to translating each component $\mathbf{c}_N^j \mapsto \tau_{-\text{sign}(\underline{r}_{N,j})} \mathbf{c}_N^j$ to the reference site $x = 0$, making them comparable with each other while being independent of the occupation at sites $x < 0$, and each depending only on the occupation at $x = 0, 1$. This is where a combinatorial mechanism arises that characterises the gradient property.

For the Porous Media Model class—namely, the interpolating model introduced here, along with the integer PMM and Bernstein models from \cite{N24}—the quantity $\nabla \mathbf{h}_N$ satisfies a recurrence relation \cite[equation (3.7)]{N24} from which the gradient property can be identified.

The growth and algebraic assumptions on the rates will be refined by additional dynamical conditions, namely, by requiring the model to belong to either Regime I or Regime II, as defined below.

\begin{Def}
	We say that the dynamics induced by $\mathbf{c}_N$ is in the Regime I if 
	\begin{align}\label{I}
		\begin{split}
			(i)\;&\exists\eta\in\Omega_N: 
			\qquad \mathbf{c}_N(\eta)
			=0,
			\\
			(ii)\;&\text{$\exists$ mobile clusters of length $\kappa^\star>1$ such that the transitions \eqref{mobility} and \eqref{mtransport} }\\&\text{occur with positive rate larger than $r_\star$, with both $k^{\star}$ and $r_\star$ independent of $N$}.
		\end{split}	
	\end{align}
	The dynamics is in the Regime II if 
	\begin{align}\label{II}\tag{II}
		\text{$\exists\;\mathfrak{m}>0$ independent of $N$ such that }\;\forall\eta\in\Omega_N,\; 
		\;\mathbf{c}_N(\eta)\geq \mathfrak{m}
		.
	\end{align}
\end{Def}

The introduction of the two previous subclasses is motivated by the slow and fast diffusion regimes in \cite{s2f}, as well as by other technical conditions. Regime \ref{I} encompasses slow diffusion, resulting from the kinetic constraints and the symmetry of the model, while Regime \ref{II} covers fast diffusion under the condition that the primitive of the diffusivity is uniformly bounded from above. 

In Regime \ref{I}, the existence of mobile clusters compensates for local configurations in which no mixing can occur. However, the rate at which these clusters move and mix must be sufficiently large for macroscopic effects to be observed. 

Regime \ref{II} considers irreducible models with rates that are, in some sense, at least as large as those of the SSEP. This permits the derivation of fast diffusion scenarios, constrained by the growth conditions \eqref{ass:rate} and \eqref{ass:rate-total}.

The model in Definition \ref{def:pmm-gen} belongs to the Regime \ref{I}, as the dynamics enjoys m.c.s of any fixed length of at least $n+2$, and the rate for a jump within each m.c. is bounded from below by a positive constant independent of $N$ for m.c.s of length \textit{larger} than $n+2$. We now connect this process with the assumptions above.
\begin{Lemma}\label{lem:grad}
Fixed $\eta\in\Omega_N$, for each $n\in\mathbb{N}$ let $P_{n}(\eta)$ be the number of particles in the box $\llbracket0,n\rrbracket$. There is $\mathbf{H}_N^m:\Omega_N\to\mathbb{R}$ such that $\mathbf{c}_{N}^{m}(\tau_x\eta)(\eta(x+1)-\eta(x))=\nabla\mathbf{H}_{N}^m(\tau_x\eta)$, for each $x\in\T_N$. 

The map $\mathbf{H}_N^m$ is given, explicitly, by $\mathbf{H}_N^m=\mathbf{h}_{\ell_N}^m+\mathbf{g}_{\ell_N}^m$ where
\begin{align}\label{h}
	\mathbf{h}_{\ell_N}^m(\eta)
	&=
	\mathbf{v}_{n}(\eta)
	+
	\sum_{k=1}^{\ell_N}
	\binom{m}{k}(-1)^k
	\big(
	\mathbf{v}_{n}(\eta)-\mathbf{v}_{n,k}\eta)
	\big)
	\\\label{g}
	\mathbf{g}_{\ell_N}^m(\eta)
	&=\frac{1}{n+k+1}\sum_{j=1}^{n+k}
	\sum_{i=1}^{j}
	\mathbf{c}_N^{m,j}(\tau_{-i}\eta)
	(\mathbf{e}_{0,1}(\tau_{-i}\eta)
	-\mathbf{e}_{1,0}(\tau_{-i}\eta)
	)
\end{align}
and
\begin{align}\label{v}
\begin{split}
	\mathbf{v}_{n}(\eta)
	&=\frac{1}{n+1}\prod_{i=0}^{n}\eta(i)
	,\\
	\mathbf{v}_{n,k}(\eta)
	&=
	\frac{1}{n+k+1}
	\sum_{\nu=n+1}^{n+k}
	\frac{\binom{\nu}{n}}{\binom{n+k}{k}}
	\mathbf{1}{
		\big\{
		P_{n+k}(\eta)
		\geq\nu+1
		\big\}
	}
	,
	\qquad\text{for each}\quad
	1\leq k\leq \ell_N
	.
\end{split}
\end{align}
\end{Lemma}
The expression for $\mathbf{H}_N$ results from collecting the expressions involved in the gradient property for the integer PMM (see, for instance, \cite[Lemma 2.23]{s2f}) and the Bernstein model \cite[Proposition 2.7]{N24}. In this way, the decomposition \eqref{ass:h-form} is satisfied. The remaining conditions are encapsulated into the next lemma.
\begin{Lemma}\label{lem:c-h-bound}
The model in Definition \ref{def:pmm-gen} satisfies \eqref{ass:h-bound}, \eqref{ass:rate},\eqref{ass:rate-total}, \eqref{ass:h-reg} and \eqref{ass:h-bound}. 
\end{Lemma} 
\begin{proof}
We start by recalling from \cite[Lemma A.1]{s2f} that for any $ k\geq 2 $ it holds that
\begin{align}\label{bin:bound}
	\abs{\binom{m}{k}}
	<
	\frac{|\Gamma(m+1)\sin(\pi(k-m-1))|}{\pi(k-m-1)^{m+1}}
\end{align}
with $\Gamma$ the Gamma-function.

We show \eqref{ass:rate} and \eqref{ass:rate-total}. Recalling the constraints as in \eqref{c}, we can bound from above 
	\begin{align}
		|\mathbf{c}_{N}^m|_{\infty}
		&\leq 
		1+2\sum_{k=1}^{\ell_N}
		\abs{\binom{m}{k}}
        ,
	\end{align}
therefore the upper bound in \eqref{bin:bound} concludes the proof. For \eqref{ass:rate-total}, we compute
\begin{align}
	\bigg|
	\sum_{j=0}^{N}
	\sum_{z\in\mathcal{R}_{N,j}}
	\tau_z\mathbf{r}_{N}^j
	\bigg|_{\infty}
	&\leq
	2n
	+
	\sum_{k=1}^{\ell_N}
	\abs{\binom{m}{k}}
	\bigg\{
	\sum_{i=-n}^{n+2}
	\big|
	\tau_{i}\mathbf{p}_{n}
	\big|_{\infty}	
	+
	\sum_{i=-(n+k)}^{n+k+2}
	\big|
	\tau_{i}\mathbf{p}_{n,k}
	\big|_{\infty}	
	\bigg\}
	\\&\leq
	4n+2
	+\sum_{k=1}^{\ell_N}
	\abs{\binom{m}{k}}
	(n+k+2),
\end{align}	
and because $\ell_N\leq N$ and $m\in(0,1)$, the quantity in the right-hand side above is of order $o(N)$.	

To see \eqref{ass:h-reg} and \eqref{ass:h-bound}, from the Hockey-Stick identity one can estimate $|\mathbf{v}_{n,k}|\leq \tfrac{1}{n+1}$, thus 
	\begin{align}\label{h:bound}
		|
		\mathbf{h}_{\ell_N}
		|_\infty
		&\leq 
		2\sum_{k=1}^{\ell_N}
		\abs{\binom{m}{k}},
		\\
		|
		\mathbf{h}_{\ell_N}^m
		-
		\mathbf{h}_{L}^m
		|_\infty
		&\leq 
		2\sum_{k=L+1}^{\ell_N}
		\abs{\binom{m}{k}},
		&&\text{for $L<\ell_N$},
	\end{align}
and to conclude it is enough to invoke \eqref{bin:bound}, and note that from an integral comparisson the quantity in the last line above is of order $O(|1/L-1/\ell_N|)$.
\end{proof}

\subsection{Topological setting}\label{subsec:topo}
We now present the analysis context of this work. Fix a finite time horizon $[0,T]$, let $\mu_N$ be an initial probability measure on $\Omega_N$, and let $\{\eta_{N^2t}\}_{t\geq 0}$ be the process generated by $\mathfrak{L}_{N}$ where, for $\mathcal{L}_N$ in the Regime \ref{I},
\begin{align}
	\mathfrak{L}_N=
	N^2\mathcal{L}_N
	+
	\mathfrak{p}_NN^2\mathcal{L}_N^{\text{SSEP}},
	&&
	\begin{cases}
		\mathfrak{p}_NN^2\rightarrow +\infty,\\
		0<\mathfrak{p}_N\rightarrow 0,
	\end{cases}
	\quad
\text{as } N\to+\infty,
\end{align}
with $\mathcal{L}_N$ as in \eqref{generic:gen} and $\mathcal{L}_N^{\text{SSEP}}$ the generator of the Symmetric Simple Exclusion process (corresponding to \eqref{generic:gen} for $\mathbf{c}(\eta)=1$). For $\mathcal{L}_N$ in the Regime \ref{II}, $\mathfrak{L}_N$ is defined analogously with $\mathfrak{p}_N=0$.

The object connecting the microscopic and macroscopic scales is the \textit{empirical measure}, \ndef{$ \pi^N $}, the random measure given by 
\begin{align*}
    \pi^N(\cdot,\mathrm{d}u)
    =\frac1N \sum_{
    	x\in\T_N}
    \delta_{\tfrac{x}{N}}
    (\mathrm{d}u)\pi_x(\cdot)
    ,
\end{align*}
where $ \delta_{v} $ is the Dirac measure at $v\in\mathbb{T}$. For each $\eta\in\Omega_N$ fixed, its (diffusive) time evolution is defined as $ \pi^N_t(\eta,\rmd u)=\pi^N(\eta_{N^2t},\rmd u)$;  and for any function $ G:\mathbb{T}\to\mathbb{R} $ we shorten the integral of $ G $ with respect to the empirical measure as
\begin{align}\label{int:emp}
		\pi_t^N(\eta,G)
		=\int_{\mathbb{T}}G(u)\pi_t^N(\eta,\rmd u)
		.
\end{align}

We now recall standard definitions and introduce notation. Denote by $\ndef{\mathcal{M}_+}$ the space of positive measures on $[0,1]$ with total mass at most $1$ and endowed with the weak topology. The Skorokhod space of trajectories induced by $\{\eta_{N^2t}\}_{t\in[0,T]}$ with initial measure $\mu_N$ is denoted by $\ndef{\mathcal{D}([0,T],\Omega_N)}$, and we denote by $\ndef{\mathbb{P}_{\mu_N}}$ the induced probability measure on it. Moreover, $\ndef{\mathbb{Q}_N}:=\mathbb{P}_{\mu_N}\circ(\pi^N)^{-1}$ is the probability measure on $\ndef{\mathcal{D}([0,T],\mathcal{M}_+)}$ induced by $\{\pi^N_t\}_{t\in[0,T]}$ and $\mu_N$. For $f,g\in L^2(\mathbb{T})$, we denote by $\ndef{\langle f,g\rangle}$ their standard Euclidean product in $L^2(\mathbb{T})$ and $\ndef{\|\cdot\|_{2}}$ its induced norm.

We now aim to introduce the relevant weak formulation of the hydrodynamic equation considered here. To that end, for any pair $G,H\in C^\infty(\mathbb{T}^d)$ let $\ndef{\inner{G,H}_1}=\inner{\partial_u G,\partial_u H}$ be their semi inner-product on $C^\infty(\mathbb{T})$, and $\ndef{\norm{\cdot}_1}$ its associated semi-norm. The space $\ndef{\mathcal{H}^1(\mathbb{T})}$ is the Sobolev space on $\mathbb{T}$, defined as the completion of $C^\infty(\mathbb{T})$ for the norm $\ndef{\norm{\cdot}_{\mathcal{H}^1(\mathbb{T})}}^2=\norm{\cdot}_{L^2}^2+\norm{\cdot}_{1}^2$. We write as $\ndef{L^2([0,T];\mathcal{H}^1(\mathbb{T}))}$ the set of measurable functions $f:[0,T]\to\mathcal{H}^1(\mathbb{T})$ such that $\int_0^T\norm{f_s}^2_{\mathcal{H}^1(\mathbb{T})}\,\mathrm{d}s<\infty$.

The next definition corresponds to the notion of weak solution, and will be associated with the Regime \ref{I}.

\begin{Def}[Weak solution I]\label{def:weak}
    For $ \rho^{\rm ini}:\mathbb{T}^d\to[0,1] $ a measurable function, we say that $ \rho:[0,T]\times \mathbb{T}\mapsto [0,1] $ is a weak solution of the equation
        \begin{align}\label{hydro-eq}
		\begin{cases}
		    \partial_t\rho
            =\partial_u^2\Phi(\rho), & \text{in } (0,T]\times\T,\\
            \rho_0=\rho^{\text{ini}},& \text{in }\mathbb{T}
		\end{cases} 
	\end{align}
    if
	\begin{enumerate}
        \item $ \Phi(\rho)\in L^2([0,T];\mathcal{H}^1(\mathbb{T})) $;
		\item for any $ t\in[0,T] $ and $ G\in C^{1,2}([0,T]\times \mathbb{T}) $, $\rho$ satisfies the formulation $\mathfrak{F}_t(\rho^{\rm ini},\rho,G)=0$, where   
		\begin{equation}\label{weak}
			\mathfrak{F}_t(\rho^{\rm ini},\rho,G)
			:=
			\inner{\rho_t,G_t}-\inner{\rho^{\rm ini},G_0}
			-\int_0^t
			\bigg\{
			\inner{\rho_s,\partial_sG_s}
			+
                \inner{\Phi(\rho_s),\partial_u^2G_s}
			\bigg\}
			\rmd s
			.
		\end{equation}
	\end{enumerate}
\end{Def}
The next notion of solution is also known as very weak solution \cite{vazquez}.
\begin{Def}\label{def:very-weak}(Weak solution II)
	For $ \rho^{\rm ini}:\mathbb{T}^d\to[0,1] $ a measurable function, we say that $ \rho:[0,T]\times \mathbb{T}\mapsto [0,1] $ is a very weak solution of the equation \eqref{hydro-eq} if $\rho\in L^2([0,T],\mathcal{H}^1(\T))$ and (2) in the previous definition are satisfied.
	
\end{Def}

\subsection{Main Result}\label{subsec:hl}
The hydrodynamic limit establishes a weak law of large numbers. Precisely, that starting from a \textit{local equilibrium distribution}, the empirical measure converges weakly to an absolutely continuous measure, whose density is the unique solution of the hydrodynamic equation \eqref{hydro-eq}. Indirectly, this establishes the existence of solutions of \eqref{hydro-eq} in the sense of Definitions \ref{def:weak} and \ref{def:very-weak}. We aim to present this statement rigorously.
\begin{Def}[Local equilibrium distribution]\label{def:ass}
	Let $ \{\mu_N\}_{N\geq 1} $ be a sequence of probability measures on $ \Omega_N $, and let $g:\mathbb{T}\to[0,1] $ be a measurable function. If, for any continuous function $ G:\mathbb{T}\to\mathbb{R} $ and every $ \delta>0 $, it holds
	\begin{align*}
		\lim_{N\to+\infty}\mu_N
		\left(
		\eta\in\Omega_N
		:\big|\pi^N(G)-\inner{g,G}\big|>\delta
		\right)
		=0,
	\end{align*}
	we say that the sequence $ \{\mu_N\}_{N\geq 1} $ is a local equilibrium measure associated with the profile $ g $.
\end{Def}

\begin{Th}[Hydrodynamic limit] \label{th:hydro}
	Let $ \rho^{\rm ini}:\mathbb{T}\to[0,1] $ be a measurable function and let $ \{\mu_N\}_{N\geq 1} $ be a local equilibrium measure associated with it.
	Then, for any $ t\in [0,T] $ and $ \delta>0 $, it holds
	\begin{align*}
		\lim_{N\to+\infty}
		\mathbb{P}_{\mu_N}
		\left(
		\big|\pi_t^N(G)-\inner{\rho_t,G}\big|>\delta
		\right)=0,
	\end{align*}
	then $ \rho $ is the unique solution of the hydrodynamic equation \eqref{hydro-eq} with initial data $\rho^{\rm ini}$, in the sense of Definition \ref{def:weak}, if $\mathcal{L}_N$ belongs to the Regime \ref{I}; or in the sense of Definition \ref{def:very-weak}, if $\mathcal{L}_N$ belongs to the Regime \ref{II}.
\end{Th}
\subsubsection{Comments on the proof}
The proof of Theorem \ref{th:hydro} introduces some novelties compared to previous works. In \cite{s2f}, technical difficulties arose due to the constraints (see \eqref{c:s2f}) depending on $N$ through $\ell_N$. A question posed and resolved there was whether $\ell_N$ needs to satisfy some growth condition. The short answer is no. To show this, the authors perform several truncations of the summation in \eqref{c:s2f} at different steps of the proof, replacing $\ell_N$ with specific functions at the macroscopic and microscopic levels in order to "slow down" the divergence $\ell_N \to \infty$ at critical points in the argument.

This alone is not enough to allow $\ell_N$ to grow without constraint. To complement this, in \cite[Lemma 3.6]{s2f}—the counterpart of \eqref{eq:psi-eps} in the present work—a parameter is introduced and subsequently fixed in the application of the so-called "2-blocks estimate". This corresponds precisely to replacing, in \eqref{eq:psi-eps}, $\varepsilon$ by $\varepsilon/m$ for each $m$.

This strategy works because the tail of the series representation of $\rho^{m-1}$ is sufficiently light. This is further explored in \cite{CG24}, where the authors provide more precise information for diffusion coefficients expressible in series form, beginning by identifying the PMM with a monomial basis.

In \cite{N24}, it was explained that PMMs are inadequate for modelling diffusivities with roots of high multiplicity, leading to a class of diffusivities that cannot be obtained from previous works, which rely on the PMM as a starting point. To address this gap, \cite{N24} introduced the Bernstein model, \eqref{rate:b}. One can then see that the previous arguments used to prove the 2-block estimate (forthcoming Lemma \ref{lem:2block}) must be extended to account for diffusivities that degenerate at both $0$ and $1$. 

Our goal, therefore, was to decouple the proof of Theorem \ref{th:hydro} from the choice of basis, and to generalise the arguments in the 2-block estimate to encompass other types of diffusion. To do so independently of the specific expression of the rates, we identified the necessary properties more generally. This led to the simpler technical conditions stated in Assumption \ref{ass}, along with dynamical conditions of either irreducibility or the presence of mobile clusters.

\section{Proof of the Hydrodynamic Limit}\label{sec:hl}

In order to prove Theorem \ref{th:hydro}, we apply the entropy method, first introduced in \cite{entropy}, following the method as presented in \cite{s2f}. Let us overview the approach.

\subsection{General scheme}
The link between our process and the weak formulation \eqref{weak} is given through Dynkin's martingale (see \cite[Appendix 1, Lemma 5.1]{KL:book}),
\begin{align}\label{eq:M}
	\text{M}_N^G(t)
	:=\pi_t^N(\eta,G_t)-\pi_0^N(\eta,G_0)
	&-\int_0^t
	\pi_s^N(\eta,\partial_sG_s)+\mathcal L[\pi_s^N](\eta,G_s)\rmd s
\end{align}
for $G\in C^{2,1}(\mathbb{T}\times [0,T])$. Indeed, one can show that the sequence of probability measures $ (\mathbb{Q}_N)_{N\in\mathbb{N}} $ is tight with respect to the Skorokhod topology of $ \mathcal{D}\left([0,T],\mathcal{M}_+\right) $ by resorting to Aldous' conditions \cite[proof of Proposition 4.1]{GMO22}, that can be shown to be satisfied due to the the quadratic variation of $\text{M}_N^G(T)$ vanishing as $N\to\infty$, which is consequence of Assumption \ref{ass}. This is analysed in Subsection \ref{subsec:tight}. We present the proof for the Regime \ref{I}, but we note that in Regime \ref{II}, that is, withouth the perturbation term $\mathfrak{p}_N$, is identical.

With this, we conclude that the sequence of empirical measures is tight, existing then weakly convergent subsequences. Since there is at most one particle per site, one can show that the limiting measure of a convergent subsequence of $(\mathbb{Q}_N)_{N\geq 0}$, that we write as \ndef{$\mathbb{Q}$}, is concentrated on paths of absolutely continuous measures with respect to the Lebesgue measure. This means precisely that the sequence $ (\pi_\cdot^{N}(\eta,\rmd u))_{N\in\mathbb{N}} $ converges weakly, with respect to $ \mathbb{Q}_N $, to an absolutely continuous measure with density \ndef{$\rho$}, that is, $ \pi_{\cdot}(\rmd u)=\rho_\cdot(u)\rmd u $. 

Provided the aforementioned convergence by subsequences and absolutely continuity of the limit measure, we now argue that the density $\rho$ satisfies the formulation in Definition \ref{def:weak}. This encompasses two results to be held $\mathbb{Q}$ almost-surely: \textit{\textbf{(i)}} satisfaction of the weak formulation $\mathcal{F}_t(\rho^{\text{ini}},\rho,G)=0$; and \textit{\textbf{(ii)}} regularity: $\Phi(\rho)\in L^2([0,T];\mathcal{H}^1(\T))$, for the Regime \ref{I}, while $\rho\in L^2([0,T];\mathcal{H}^1(\T))$ for the Regime \ref{II}. In this work we present the proofs for the Regime \ref{I} only, being the more technically involved. Regime \ref{II} accounts for many simplifications, that are going to be explained in Remarks \ref{rem:fde-char} and \ref{rem:fde-energy}.

This proves that for any limit point $ \mathbb{Q} $ of $ (\mathbb{Q}_N)_{N\in\mathbb{N}} $ it holds, for the Regime \ref{I}, 
    \begin{align}
    \mathbb{Q}
    \bigg(\pi_{\cdot}\in\mathcal{D}([0,T],\mathcal{M}_+) :
    \;
        \begin{cases}
    	\Phi(\rho)\in L^2([0,T];\mathcal{H}^1(\mathbb{T})) 
            \\
            \mathfrak{F}_t(\rho^{\rm ini},\rho,G) =0 
            ,\;
            \forall t\in[0,T],\;
            \forall G\in C^{1,2}([0,T]\times\mathbb T)
        \end{cases}	\bigg)
	=1,
    \end{align}
where  $\mathfrak{F}_t(\rho^{\rm ini},\rho,G)$  is given in  \eqref{weak}; and also in the Regime \ref{II} the above with $\Phi(\rho)$ replaced by $\rho$. The subsections \ref{subsec:tight} and  \ref{subsec:char} rely on estimates provided in Section \ref{sec:rep-lem}, the so-called replacement lemmas.

Concluding the proof, the regularity of either $\rho$ or $\Phi(\rho)$, depending on the regime, implies that there is at most one weak solution, meaning that the previous convergence by subsequences can be taken as a complete convergence. For the Regime \ref{I} this is very straightforward by Oleinik's method (see, for instance, \cite[Subsection $7.1$]{BDGN}), and for the Regime \ref{II} it is more involved, but completely analogous to the uniqueness result in the fast diffusion regime of \cite[Appendix B]{s2f} by recalling that $(\Phi_L)_L$, as defined in \eqref{phiL}, converges uniformly to $\Phi$, and noting that because the dynamics is discrete, of exclusion and gradient types, $\Phi_L(\rho)$ is a non-increasing polynomial in $\rho\in[0,1]$. 

\subsection{Tightness}\label{subsec:tight}
In the present subsection we are going to show the following.
\begin{Prop}
    The sequence of measures $(\mathbb{Q}_N)_{N\geq 1}$ is tight.
\end{Prop}
    
    \begin{proof}
		The argument is given by an application of Aldous' conditions. We direct the reader to \cite{s2f,BDGN} and \cite[proof of Proposition 4.1]{GMO22} for a more detailed exposition, and focus here on the aspects related with our specific model. 
        
        From the standard machinery, it is enough to show that for all $ \eps>0 $ 
		\begin{align}
			\limsup_{\gamma\to0}\limsup_{N\to+\infty}
			\mathbb{P}_{\mu_N}
			\left(\eta_\cdot:\;
			\sup_{|t-s|\leq\gamma}
			\abs{
				\pi^N_t(\eta,G)-\pi^N_s(\eta,G)
			}>\eps
			\right)
            =0,
		\end{align}
		where $ G $ is in a dense subset of $ C([0,1]) $ with respect to the uniform topology. Recalling the martingale $ \text{M}_N^G $ in \eqref{eq:M}, the limit above can be guaranteed by the following two:
		\begin{align}\label{aldous0}
			&\lim_{\gamma\to0}\limsup_{N\to+\infty}
			\mathbb{P}_{\mu_N}
            \bigg(\eta_{\cdot}:\;
			\sup_{|t-s|\leq\gamma}
			\abs{
				\text{M}_N^G(t)-\text{M}_N^G(s)
			}
			>\frac{\eps}{2}
		\bigg)
            =0,
            \\
            &\lim_{\gamma\to0}\limsup_{N\to+\infty}
			\mathbb{P}_{\mu_N}
            \bigg(
            \eta_{\cdot}:\;
			\sup_{|t-s|\leq\gamma}
            \abs{
				\int_s^t
				\mathfrak{L}[\pi_{r}^N](\eta,G)\rmd {r}
			}
			>\frac{\eps}{2}
		\bigg)
        =0.
		\end{align}
		From the triangle, Jensen's and Doob's inequalities for the the first term above, and Proposition \ref{prop:pedro} for the second, it is enough show show that 
		\begin{align}\label{aldous}
\begin{split}
    			\limsup_{N\to+\infty}
			\mathbb{E}_{\mu_N}\left[
			\left(
			\text{M}_N^G(T)
			\right)^2
			\right]^\frac12 &=0
,\\
			\lim_{\gamma\to0}\limsup_{N\to+\infty}
			\mathbb{E}_{\mu_N}\left[
			\abs{
			\mathbf{1}_{|t-s|\leq \gamma} 	\int_s^t
				\mathfrak{L}[\pi_{r}^N](\eta,G)\rmd {r}
			}
			\right]&=0.
\end{split}
		\end{align}
From \cite[A1 Lemma 5.1]{KL:book}, $\text{N}_N^G$ is a martingale, where
\begin{align}
\text{N}_N^G(t)=
\left(
\text{M}_N^G(t)
\right)^2
-\int_0^{t}\text{F}_N^G(s)\rmd s
\end{align}
and
\begin{align*}
\text{F}_N^G(s)
&=
\sum_{x\in\mathbb{T}_N}
(\mathfrak{p}_N+\tau_x\mathbf{r}_{N}(\tau_x\eta_{N^2s}))
\left(G_s(\tfrac{x+1}{N})-G_s(\tfrac{x}{N})\right)^2
\end{align*}
By assumption \eqref{ass:rate}, $|\mathbf{c}_{N}|_{\infty}=o(N)$ and in particular
\begin{align}
	\text{F}_N^G(s)
	\leq
	\frac{|\mathbf{c}_{N}|_{\infty}}{N}
	\frac1N\sum_{x\in\mathbb{T}_N}
	\left(NG_s(\tfrac{x+1}{N})-NG_s(\tfrac{x}{N})\right)^2
	\xrightarrow{N\to+\infty}0,
\end{align}
guaranteeing the first equality in \eqref{aldous}.

For the second, the gradient property and the periodic boundary imply directly that
\begin{align}
    \mathfrak{L}[\pi_{r}^N](\eta,G)
    =\frac1N\sum_{x\in\T_N}
    \Delta_NG(\tfrac{x}{N})
    (
    \mathbf{H}_{N}(\tau_x\eta_{N^2r})
    +\mathfrak{p}_N\eta_{N^2r}(x)
    )
\end{align}
with $\mathbf{H}_{N}=\mathbf{h}_{N}+\mathbf{g}_{N}$ as in \eqref{ass:h-form}. By hypothesis, $|\mathbf{h}_{N}|_\infty$ is uniformly bounded from above by a constant independent of $N$, as in \eqref{ass:h-bound}, and in this way is enough to estimate the term associated with $\mathbf{g}_{N}$, since $\mathfrak{p}_N\xrightarrow{N\to\infty}0$. The sequence $(|\mathbf{g}_{N}|_{\infty})_{N\geq 1}$ can be divergent, therefore it needs to be analysed in a different way. From an application of Proposition \ref{prop:pedro} and then Lemma \ref{lem:rest}, that yields
    \begin{align}\label{eq:tight-g}
        \mathbb{E}_{\mu_N}
        \bigg[
        \bigg|
        \int_{0}^t
        \frac1N \sum_{x\in\T_N}\Delta_NG(\tfrac{x}{N})
        \mathbf{1}_{\{r>s,\;|t-s|\leq \gamma\}}
        \mathbf{g}_{N}(\tau_x\eta_{N^2r})
        \rmd r
        \bigg|
        \bigg]
        \leq 
        c\norm{G}_{\infty}\sqrt{\gamma}
        \sqrt{\frac{\underline{r}_N }{N}}
    \end{align}
with $c>0$ a positive constant independent of $G,s,t,\gamma$ and $N$, and $\underline{r}_N$ as in \eqref{total-rate}. Recalling \eqref{ass:rate-total}, in order to conclude the proof we perform the limits $\lim_{\gamma\rightarrow 0}\limsup_{N\rightarrow+\infty}$. 

\end{proof}

\subsection{Characterization of the limit points}\label{subsec:char}
The main result of this subsection is the following
\begin{Prop}\label{prop:char}
	For any limit point $ \mathbb{Q} $ of a convergent subsequence of $(\mathbb{Q}_N)_{N\in\mathbb{N}} $ it holds 
\begin{align}
        \mathbb{Q}
	\big(\pi_{\cdot} 
        :\;
        \mathfrak{F}_t(\rho^{\rm ini},\rho,G) =0 
        ,\;
            \forall t\in[0,T],\;
        \forall G\in C^{1,2}([0,T]\times\mathbb T)
        \;\big|\;
        \tfrac{\pi(\rmd u)}{\rmd u}=\rho
        \big)
	=1
    ,
\end{align}
where  $\mathfrak{F}_t$ is given in  \eqref{weak}.
\end{Prop}

\begin{proof}
For the rest of this proof we are going to assume the absolute continuity $\pi(\rmd u)=\rho(u)\rmd u$, even though we do not write its conditioning explicitly. We are going to show that for any $ \delta>0 $, 
\begin{align}\label{q:prob_equiv}
		\sup_{t\in[0,T]}
		\bigg|
	\Big\langle
        G_t
        ,\rho_t
        \Big\rangle
        -\Big\langle
    	G_0
            ,\rho^{\rm ini}
        \Big\rangle
	-\int_0^t
            \Big\langle
    	    \partial_sG
                ,\rho
            \Big\rangle
        \rmd s
	-
        \int_0^t
        \Big\langle
		\partial_{u}^2 G
            ,
            \Phi(\rho)
	\Big\rangle
        \rmd s
    \bigg|
		\leq \delta
  &&
  \mathbb{Q}-\text{a.s.},
\end{align}
from which Proposition \ref{prop:char} then follows. Recalling \eqref{phiL}, let us fix $L$ large enough, that shall be taken to $+\infty$ in the end, so that 
\begin{align}
	\norm{\partial_{u}^2 G}_{\infty}
	\int_0^t
	\int_\T 
	\big|
	\Phi(\rho_s(u))
	-
	\text{E}_{\nu_{\rho_s(u)}^{L^\star}}[\mathbf{h}_{L}]
	\big|
	\rmd u
	\rmd s
	<\delta/2,
\end{align}
and, writing \ndef{$\mathfrak{F}^L_t$} for the quantity inside the absolute value in \eqref{q:prob_equiv} with $\Phi$ replaced by $\Phi_L(\rho)$ for some fixed $L$, we will argue that 
\begin{align}\label{char:eq1}
    \mathbb{Q}
    (
    \pi_\cdot:\:
    \sup_{t\in[0,T]}
    \big|
    \mathfrak{F}_t^L(\rho^{\rm ini},\rho,G)
    \big|
    >\delta/2
    )
    =0
\end{align}
by applying Portmanteau's theorem, and then proceeding with arguments at the discrete level. Because the map $\pi_\cdot\mapsto \sup_{t\in[0,T]}
    \big|
    \mathfrak{F}_t^L(\rho^{\rm ini},\rho,G)
    \big|$ is not continuous with respect to the Skorokhod topology, this cannot be performed directly (see, for example, \cite{s2f,BMNS} for more details). It is standard to proceed with a smoothing argument as we now explain. Let us consider $\tilde{\Phi}_L$ to be the extension of $\mathbf{h}_{L}\equiv\mathbf{h}_{L}(\eta(0),\eta(1),\dots,\eta(N-1))$ given by $\ndef{\tilde{\Phi}_L}[f]:=\mathbf{h}_L(f(0),f(1),\dots,f(N-1))$, for any $f:\T_N\to\mathbb{R}$, where we extended the domain of $\mathbf{h}_{L}$ from $\{0,1\}^{\T_L^\star}$ to $\mathbb{R}^{\T_{L^\star}}$, that is, $v\in\mathbb{R}^{\T_{L^\star}}\mapsto\mathbf{h}_L(v_0,v_1,\dots,v_{L^\star-1})$, and therefore perform the abuse of notation $f\equiv (f(x))_{x\in\T_L^\star}$ in writing $\tilde{\Phi}_L[f]$. Concretely, $\tilde{\Phi}_L[f]$ corresponds to replacing $\eta(x)$ in the expression of $\mathbf{h}_{L}(\eta)$ by $f(x)$, for each $x\in\T_L^\star$. In particular, $\tilde{\Phi}_L[\mathbf{1}_\alpha]=\Phi_L(\alpha)$ with the constant map $x\in\T_N\to\mathbf{1}_\alpha(x)=\alpha\in[0,1]$. The main observation is that $\tilde{\Phi}_L$ is a "linearization" of $\Phi_L$ (the quantity $\tilde{\Phi}[\rho_s((\cdot)/L^\star)]$ is linear on any element of $\{\rho(x/L^\star)\}_{x\in\T_L^\star}$) provided by the dynamics.

Fixed $ \eps>0 $ and $ u\in\T $, define the cut-off function \ndef{$\iota_\eps^{u}$} through
\begin{align}
    {\iota}_\eps^{u}(v)
    :=
    \frac{1}{\eps}\mathbf{1}_{B_{\eps}(u)}(v), 
\end{align}
for each $v\in\T$, where $\ndef{B_\eps}(u)=[u,u+\eps)$, and where $\ndef{\mathbf{1}_A}(v)=1$ if $v\in A$ and zero otherwise. Let \ndef{$\Tilde{\iota}_\eps$} be a mollifier and such that for each $s\in[0,T]$ and for a.e. $u\in\T^d$ it holds that
\begin{align}\label{eq:moll}
    |\pi_s(\Tilde{\iota}_\eps\star\iota_\eps^u)
    -\pi_s(\iota_\eps^u)
    |<\eps
    .
\end{align}
Shortening $\ndef{\Psi_u^{\eps,L}}[\pi_s]\equiv\Tilde{\Phi}_L[\pi_s(\iota_\eps^{u+(\cdot)\eps})]$ and $\ndef{\Psi_u^{\eps,L,\star}}[\pi_s]\equiv\Tilde{\Phi}_L[\pi_s(\Tilde{\iota}_\eps\star\iota_\eps^{u+(\cdot)\eps})]$, we claim that
\begin{align}
    \big|
    \Psi_u^{\eps,L}[\pi_s]
    -\Psi_u^{\eps,L,\star}[\pi_s]
    \big|
    &\leq 
    c_{L}
    \eps,
   \label{eq:psi-star}\\
    \big|\Phi_L(\rho_s(u))
    -\Psi_u^{\eps,L}[\pi_s]\big|
    &\leq
    c_{L}'
    \sum_{m=1}^{L+1}|\rho_s(u)-\pi_s(\iota_{m\eps}^{u})|
    ,
    \label{eq:psi-eps}
\end{align}
for $c_{L},c_{L}'$ constants dependent on $L$. From Lebesgue's differentiation theorem, $\lim_{\eps\to0}|\rho_s(u)-\pi_s(\iota_{m\eps}^u)|=0$ for almost every $u\in\T$ and every $s\in [0,T]$ and $1\leq m\leq L+1$.

To prove the claim, we see that for any sequence of real numbers $(a_i)_i,(b_i)_i$, one can rearrange
\begin{align}\label{reorg}
    \prod_{i=0}^La_i
    -\prod_{i=0}^Lb_i
    =
    \sum_{m=0}^{L}(a_m-b_m)\prod_{\substack{i=0\\i\neq m}}^{L}c_i
    \where
    c_i=\begin{cases}
        a_i, &i< m,\\b_i, &i>m.
    \end{cases}
\end{align} 
In this way, \eqref{eq:psi-star} follows the above, from bounding $\pi_s(\iota_\eps^{u+p\eps})\leq 1$ and $\pi_s(\Tilde{\iota}_\eps\star\iota_\eps^{u+p\eps})\leq 1$ for each $p\in\T_N$, and \eqref{eq:moll}. In order to see \eqref{eq:psi-eps}, we also use \eqref{reorg}, bound from above $\rho(u),\pi_s(\iota_\eps^{u+p\eps})\leq 1$, and finally, estimate
\begin{align}
\begin{split}
    \big|
    \rho_s(u)-\frac{1}{\eps}\int_{u}^{u+m\eps}\rho(v)\rmd v
    \big|
    &\leq 
    (m+1)
    \bigg|
	\rho_s(u)-\frac{1}{(m+1)\eps}\int_{u}^{u+(m+1)\eps}\rho_s(v)\rmd v
	\bigg|
\\&+
   m\bigg|
				\rho_s(u)
				-\frac{1}{m\eps}\int_{u}^{u+m\eps}
				\rho_s(v)\rmd v
    \bigg|
    .
    \end{split}
\end{align}

In this way, from the monotonicity of the limit, the probability in \eqref{char:eq1} is bounded from above by
\begin{align}
    \limsup_{L\to\infty}
    \lim_{\eps\to0}
\mathbb{Q}
\bigg(\pi_\cdot:\;
    \sup_{t\in[0,T]}
    \bigg|
    \pi_t(G_t)-\pi_0(G_0)
    -\int_0^t
    \pi_s(\partial_sG)
    \rmd s
-
    \int_0^t
    \Big\langle
    \partial_{u}^2 G(\cdot)
    ,
    \Psi_{(\cdot)}^{\eps,L,\star}[\pi_s]
    \Big\rangle
        \rmd s
    \bigg|>\tfrac{\delta}{2^3}
\bigg),
\end{align}
where we used that because $ \mu_N $ is a local equilibrium measure associated with the profile $ \rho^{\rm ini}$, then $\pi_0(G_0)=\inner{G_0
            ,\rho^{\rm ini}}$ with probability one. 

From Portmanteau's theorem, then again \eqref{eq:psi-star} while assuming that $\eps>0$ is small enough, the probability in the previous display is no larger than
\begin{multline}\label{char:eq2}
\limsup_{N\to+\infty}
\mathbb{P}_{\mu_N}
\bigg(
\eta_\cdot:\;
\sup_{t\in[0,T]}
\bigg|
    \pi_t^N(\eta,G_t)
    -\pi_0^N(\eta,G_0)
    -\int_0^t
        \pi_s^N(\eta,\partial_sG_s)
    \rmd s
\qquad\qquad\qquad\\
-\int_0^t
    \Big\langle
    \partial_u^2 G_s(\cdot)
    ,
    \Psi_{(\cdot)}^{\eps,L}[\pi_s^N(\eta)]
    \Big\rangle
\rmd s
\bigg|>\frac{\delta}{2^4}
\bigg)
\end{multline}
We stress that $\Psi_{(\cdot)}^{\eps,L}[\pi_s^N(\eta)]$ corresponds to the expression of $\mathbf{h}_{L}(\eta)$ with $\eta(z)$ replaced by $\pi_s^N(\eta,\iota_\eps^{(\cdot)+z\eps})$ for each $z\in\T_N$.

Recalling the expression for the martingale $\text{M}_N^G$ in \eqref{eq:M}, from the gradient property as in Lemma \ref{lem:grad}, the probability in \eqref{char:eq2} can be bounded from above by 
\begin{align}
&\mathbb{P}_{\mu_N}
\bigg(
\sup_{t\in[0,T]}
\bigg|
\text{M}_N^G(t)
\bigg|
>\frac{\delta}{2^5}
\bigg)
\\\label{eq:pN}+&
\mathbb{P}_{\mu_N}
\bigg(
\sup_{t\in[0,T]}
\bigg|
    \int_0^t
        \frac{1}{N}
        \sum_{x\in\T_N}
		\Delta_N G_s(\tfrac{x}{N})
         \mathfrak{p}_N\eta_{N^2 s}(x)
    \rmd s
\bigg|>\frac{\delta}{2^6}
\bigg)\\+&
\mathbb{P}_{\mu_N}
\bigg(
\sup_{t\in[0,T]}
\bigg|
    \int_0^t
        \frac{1}{N}
        \sum_{x\in\T_N}
		\Delta_N G_s(\tfrac{x}{N})
            \mathbf{g}_{\ell_N}
            (\tau_x\eta_{N^2 s})
    \rmd s
\bigg|>\frac{\delta}{2^7}
\bigg)
\\
+&
\mathbb{P}_{\mu_N}
\bigg(
\sup_{t\in[0,T]}
\bigg|
\int_0^t
    \frac1N 
    \sum_{x\in\T_N}
    \Delta_NG(s,\tfrac{x}{N})
\bigg\{
\mathbf{h}_{\ell_N}(\tau_x\eta_{N^2s})
-
\mathbf{h}_{L}(\tau_x\eta_{N^2s})
\bigg\}
\rmd s
\bigg|
>\frac{\delta}{2^8}
\bigg)
\\\label{h-to-rep}
+&
\mathbb{P}_{\mu_N}
\bigg(
\sup_{t\in[0,T]}
\bigg|
    \int_0^t
        \frac{1}{N} 
        \sum_{x\in\T_N}
            \Delta_N G_s(\tfrac{x}{N})
            \bigg\{
            \Psi_{\tfrac{x}{N}}^{\eps,L}[\pi_s^N(\eta)]
                -
                \mathbf{h}_L(\tau_x\eta_{N^2s})
            \bigg\}
    \rmd s
\bigg|
>\frac{\delta}{2^{9}}
\bigg)
\\
+&
\mathbb{P}_{\mu_N}
\bigg(
\sup_{t\in[0,T]}
\bigg|
\int_0^t
    \partial_u^2 G_s(u)
    \Psi_{u}^{\eps,L}[\pi_s^N(\eta)]
    \rmd u
\rmd s
\\&
\qquad\qquad\qquad
-\int_0^t
    \frac1N
    \sum_{x\in\T_N}
    \Delta_NG_s(\tfrac{x}{N})
                \Psi_{\tfrac{x}{N}}^{\eps,L}[\pi_s^N(\eta)]
\rmd s
\bigg|
>\frac{\delta}{2^{10}}
\bigg)
.
\end{align}

We analyse each of the terms above:
\begin{itemize}
    \item From Doob's inequality and \eqref{aldous}, the first term vanishes in the limit $N\to+\infty$; 

    \item The second vanishes since $\mathfrak{p}_N\xrightarrow{N\to\infty}0$;
    
    \item For the third, one applies Proposition \ref{prop:pedro} and the forthcoming Lemma \ref{lem:rest}. We also recall the computations just after \eqref{eq:tight-g},

	\item For the fourth, we recall Assumption \ref{ass}

	\item The last term is analysed by performing a Taylor expansion, then approximating the integral by a Riemann sum. 
\end{itemize}
The fifth term is more involved, and we analyse separately.

From here on, for each $l\in\mathbb{N}_+$ consider the ball $\ndef{B_{l}}=\llbracket 0,l-1\rrbracket$, and shorten for any $\eta\in\Omega_N$ the density $\ndef{\inner{\eta}_l}$ as   
\begin{align*}
	\inner{\eta}_{l}
    :=\frac{1}{l} \sum_{y\in B_{l}}\eta(y).
\end{align*}

We now focus on \eqref{h-to-rep}. 
For any set $A\subset\T_N$ and $P\subset A$, once can express
\begin{align}
    &\prod_{p\in P}
    \pi_s^N(\eta,\iota_{\eps}^{p\eps})
    \prod_{q\in A\backslash P}
    (1-\pi_s^N(\eta,\iota_{\eps}^{q\eps}))
    -
    \prod_{p\in P}\eta(p)
    \prod_{q\in A\backslash P}(1-\eta(q))
    \\&=
    \sum_{m\in A}
    \varphi_m^{\eps,\eps N}(\eta)
\bigg\{
    \pi_s^N(\eta,\iota_{\eps}^{m\eps})
-
    \inner{\tau_{\ceil{m\eps N}}\eta}_{\ceil{\eps N}}
\bigg\}
\label{embbeb}\\
 &+    \sum_{m\in A}
  \varphi_m^{\eps N,L}(\eta)
 \bigg\{
     \inner{\tau_{\ceil{m\eps N}}\eta}_{\ceil{\eps N}}
 - 
     \inner{\tau_{mL}\eta}_{{L}}
 \bigg\}
 \label{two-block}\\
&
+    \sum_{m\in A}
 \varphi_m^{L}(\eta)
\bigg\{
\inner{\tau_{mL}\eta}_{{L}}
 -\eta(mL)
\bigg\} 
\label{one-block}\\&
 +
     \sum_{m\in A}
    \varphi_m^{L}(\eta)
\bigg\{
\eta(mL)
 -
 \eta(m)
 \bigg\}\label{mix},
\end{align}
where, for each $m\in A$, 
\begin{itemize}
    \item $\varphi_m^{\eps N,L}$ is independent of the occupation value at $\llbracket\ceil{m\eps N},\ceil{m\eps N}+\ceil{\eps N}\rrbracket\cup\llbracket mL,mL+L\rrbracket$;

    \item $\varphi_m^{L}$ is independent of the occupation value at $\llbracket mL,mL+L\rrbracket\cup\llbracket mL,mL+L\rrbracket$;

    \item $\varphi_m^{L}$ is independent of the occupation value at $\{mL,m\}$;

    \item Each of the previous maps and $\varphi_m^{\eps,\eps N}$ are uniformly bounded from above by $1$.

\end{itemize}
The particular expressions for the "$\varphi$" functions are not important, only their boundedness and independence of the occupation-values at specific sites. In order to prove that the $\lim_{L\to\infty}\lim_{\eps\to0}\limsup_{N\to\infty}$ of the probability in \eqref{h-to-rep} is zero, we split the aforementioned probability into four probabilities, each associated with a translation by $x$ of each term of the decomposition in the previous display. For \eqref{embbeb}, from Markov's inequality it is enough to see that for each $0\leq m\leq L$,
\begin{align}
\Big|
\pi_s^N(\eta,\iota_{\eps}^{m\eps})
-
\inner{\tau_{\ceil{m\eps N}}\eta}_{\ceil{\eps N}}
\Big|
=
\Big|    \frac{1}{\eps N}
    \sum_{y\in B_{\eps N}(m\eps N)\cap \T_N}\eta(y)
    -\frac{1}{\ceil{\eps N}}
    \sum_{y\in B_{\ceil{\eps N}}(\ceil{m\eps N})}\eta(y)
\Big|
\end{align}
and that the quantity above vanishes in the limits $\lim_{\eps\to 0}\limsup_{N\to\infty}$.

Next, we apply Proposition \ref{prop:pedro} to each of the remaining probabilities and analyse them separately with the Replacement Lemmas proved in Section \ref{sec:rep-lem}. The terms associated with \eqref{mix} and \eqref{one-block} are analysed with the upcoming Lemma \ref{lem:1block}, with $L$ in the statement of the aforementioned lemma either equal to $1$ or an arbitrary integer $>1$, respectively, and $w$ and $y$ chosen arrodingly. 

In order to conclude the proof, to treat the term associated with \eqref{two-block}, it is now enough to apply the triangle inequality and then invoke Lemma \ref{lem:2block}, fixing there $\ell\equiv\eps N$ and $w,y$ accordingly, noting that $\rmd(w,y)<L+\eps N$. 

\begin{Rem}[Regime II]\label{rem:fde-char}
	In this case, the diffusion is sufficiently fast that the arguments for the replacement lemmas become much simpler. Specifically, as seen in \cite{s2f} and in the proof of Lemma \ref{lem:1block}, one can follow the proof exactly as presented, obtaining the upper bound \eqref{bound:lem1block} with $\mathfrak{p}_N$ replaced by $\mathfrak{m}$, since in this regime the minimal rate is given by $\mathfrak{m}$ rather than $\mathfrak{p}_N$. This allows us to set $L = \varepsilon N$ directly in \eqref{bound:lem1block}. In particular, instead of the terms \eqref{embbeb}--\eqref{mix}, we have \eqref{embbeb} and \eqref{two-block} with $L \equiv 1$ in \eqref{two-block}.
	
	For the remainder of the characterisation of the limit points, the only further simplification is the absence of the term in \eqref{eq:pN}.
	
\end{Rem}
\end{proof}

\section{Replacement Lemmas}\label{sec:rep-lem}
The goal of this section is to prove the Lemmas \ref{lem:rest},\ref{lem:1block} and \ref{lem:2block}, in the context of Regime I. As mentioned in Remark \ref{rem:fde-char}, there is only the need to show the analogue of Lemma \ref{lem:1block} in the Regime II, with the proof being identical except for the term $\mathbf{p}_N$ being replaced by $\mathfrak{m}$.

In order to prove our results, we introduce the next objects.
\begin{Def}
Let $g:\Omega_N\to\mathbb{R}$ and $\nu$ be a probability measure on $\Omega_N$, and  introduce
    \begin{align}
        \frac12
        \Gamma_Ng
        :=
        g(-\mathfrak{L}_N)g
        +\frac12
        \mathfrak{L}_Ng,
    \quad\text{and}\quad
    \mathfrak{D}_N(g|\nu)
        :=
        \int_{\Omega_N}
        \Gamma_Ng
        \;\rmd
        \nu
        .
    \end{align}
\end{Def}
It will be convenient to make a couple of remarks. Fixed a constant profile $\alpha\in(0,1)$, the measure $\nu_\alpha^N$ will be our reference measure. We say that $f:\Omega_N\to\mathbb{R}^+\cup\{0\}$ is a \textit{density} (with respect to $\nu_\alpha^N$) if $\nu_\alpha^N(f)=1$. Because $\mathfrak{L}_N$ is reversible w.r.t. $\nu_\alpha^N$, for any density $f$ it holds that 
        \begin{align}\label{revers-dir}
        \frac12
        \mathfrak{D}_N(f|\nu_\alpha^N)
        =
        \int_{\Omega_N}
        f
        (-\mathfrak{L}_N)f
        \;\rmd\nu_\alpha^N
        .
        \end{align}
Moreover, fixed $x,y\in\T_N$, for any $\varphi:\Omega_N\to\mathbb{R}$ independent of the transformation $\eta\mapsto\theta_{x,y}\eta$ and any $g:\Omega_N\to\mathbb{R}$, it holds that
        \begin{align}\label{by-parts}
            \int_{\Omega_N}\varphi(\eta)
            (\eta(x)-\eta(y))
            g(\eta)\rmd\nu_\alpha^N
            =-\frac12\int_{\Omega_N}\varphi(\eta)
            (\eta(x)-\eta(y))
            \nabla_{x,y}
            g(\eta)\rmd\nu_\alpha^N
            .
        \end{align}

There exists a standard procedure that will be a common starting step for the proof of the forthcoming lemmas. For future reference, we summarize it in the next proposition, and overview its proof for completeness. 
\begin{Prop}\label{prop:base}
Fixed $\alpha\in(0,1)$, for any  $\mathcal{V}:\Omega_N\times [0,T]\to\mathbb{R}$ bounded there is a positive constant \ndef{$c_{\alpha}$}, independent of $N$ and $\mathcal{V}$, such that, for any $A>0$, 
\begin{align}
        \mathbb{E}_{\mu_N}\Big[
        \Big|
        \int_0^t
        \mathcal{V}_s(\eta_{N^2s})
        \rmd s
        \Big|
        \Big]
        \leq 
        \frac{c_{\alpha}}{A}
        +
        \int_0^t
        \sup_{f \text{ density}}
        \bigg\{
        \big|
        \mathcal{V}_s(\eta)
        f(\eta)\; \rmd\nu_\alpha^N(\eta)\big|
        -\frac{\mathfrak{D}_{N}(\sqrt{f}|\nu_\alpha^N)}{2N A}
        \bigg\}
        \rmd s.
    \end{align}
\end{Prop}

\begin{proof}
For $\mu$ and $\nu$ two probability measures on $\Omega_N$, the relative entropy of $\mu$ with respect to $\nu$ is defined as
    \begin{align}
        \text{H}(\mu|\nu)
        :=\sup_{g:\Omega_N\to\mathbb{R}}
        \big\{
            \mu(g)
        -\log\nu(e^g)    
        \big\}.
    \end{align}
Fixed $\alpha\in(0,1)$, the measure $\nu_\alpha^N$ is absolutely continuous w.r.t. $ \mu_N$, thus from \cite[Theorem 8.3]{KL:book}, $\text{H}(\mu_N\mid\nu_\alpha^N)=\nu_\alpha^N(g\log g)$, with $g=\frac{\rmd\mu}{\rmd\nu}$. From this it follows that because $\alpha\neq 0,1$, there exist a positive constant $c_{\alpha}$ dependent of $\alpha$ and independent of $N$, such that $\text{H}(\mu_N|\nu_\alpha^N)\leq Nc_{\alpha}$.

From the entropy inequality \cite[Page 338]{KL:book}, for any $A>0$ and $V:\Omega_N\to\mathbb{R}$ bounded, it holds that $NA\mu_N(V)\leq\text{H}(\mu_N\mid\nu_\alpha^N)+\log\nu_\alpha^N(e^{NAV})$, allowing us to shift the initial measure to the reference measure $\nu_\alpha^N$. Using the fact that $e^{|x|}\leq e^{x}+e^{-x}$ and $\lim_{N\to+\infty}\tfrac1N\log(a_N+b_N)\leq \max\{\lim_{N\to+\infty}\tfrac1N\log a_N,\lim_{N\to+\infty}\tfrac1N\log b_N\}$ and then applying Feynmann-Kac's formula as in \cite[Page 14]{BMNS}, leads then to the upper bound in the statement.

\end{proof}

In order to apply the previous result, we will recurrently invoke the next.
\begin{Prop}\cite[Lemma 4.3.2]{phd:pedro}.\label{prop:pedro}
	Assume there exists a family $ \mathcal{F} $ of functions $ F_{N,\eps}:[0,T]\times \mathcal{D}([0,T],\Omega)\to\mathbb{R} $ satisfying
	\begin{align*}
		\sup_{\substack{\eps\in(0,1),N\geq1\\s\in[0,T],\eta_{\cdot}\in\mathcal{D}([0,T],\Omega)}}
		\big|F_{N,\eps}(s,\eta_\cdot)\big|\leq M<\infty.
	\end{align*}
	Above, the interval  (0,1) for $ \eps $ is arbitrary. We also assume that for all $ t\in[0,T] $,
	\begin{align*}
		\limsup_{\eps\to0^+}\limsup_{N\to+\infty}
		\mathbb{E}_{\mu_N}
		\left[
		\abs{
			\int_0^tF_{N,\eps}(s,{\eta_\cdot})\rmd s
		}
		\right]=0.
	\end{align*}
	Then we have for all $ \delta>0 $,
	\begin{align*}
		\limsup_{\eps\to0^+}\limsup_{N\to+\infty}\mathbb{P}_{\mu_N}
		\left(
		\sup_{t\in[0,T]}
		\abs{
			\int_0^t F_{N,\eps}(s,{\eta_\cdot})\rmd s
		}
		>\delta
		\right)=0.
	\end{align*}
\end{Prop}

We now state and prove the required replacement lemmas. One can see that the proof of the next lemma follows the same reasoning as \cite[Lemma 4.6]{s2f}. 
\begin{Lemma}\label{lem:rest}
For each $x\in\T_N$ fixed and every $t\in [0,T]$, let $\varphi:[0,T]\times \T_N\to\mathbb{R}$ be such that $M_{t,\varphi}\equiv\int_0^t\norm{\varphi(s,\cdot)}_{\infty}\rmd s<\infty$. It holds that
    \begin{align}
        \mathbb{E}_{\mu_N}
        \bigg[
        \bigg|
        \int_{0}^t
        \frac1N \sum_{x\in\T_N}\varphi(s,x)
        \mathbf{g}_{\ell_N}(\tau_x\eta_{N^2s})
        \rmd s
        \bigg|
        \bigg]
        \leq 
        \sqrt{2 c_{\alpha}}M_{t,\varphi}
        \sqrt{\frac{\underline{r}_N}{tN}}
        ,
    \end{align}
with $\mathbf{g}_{\ell_N}$ as in \eqref{g} and
\begin{align}\label{total-rate}
    \underline{r}_N
    :=
    \bigg|
    \sum_{j=0}^{\ell_N}
    \sum_{z\in\mathcal{R}_{N,j}}
    \tau_z\mathbf{r}_{N}^j
    \bigg|_{\infty}.
\end{align}
\end{Lemma}
\begin{proof}
    Proceeding as in the previous replacement lemmas, it is enough to estimate 
    \begin{align}\label{rest:var}
        \frac{c_{\alpha}}{A}
        +
    \int_0^t
        \sup_{f \text{ density}}
        \bigg\{
        \frac1N \sum_{x\in\T_N}\varphi(s,x)
        \int_{\Omega_N}
        \tau_x\mathbf{g}_{\ell_N}
        f\rmd\nu_\alpha^N
        -\frac{\mathfrak{D}_{N}(\sqrt{f}|\nu_\alpha^N)}{2N A}
        \bigg\}
    \rmd s
        .
    \end{align}
From the observation \eqref{by-parts} and Young's inequality,
\begin{align}
&\bigg|
    \frac1N \sum_{x\in\T_N}\varphi(s,x)
    \int_{\Omega_N}
    \tau_x\mathbf{g}_{\ell_N}
    f\rmd\nu_\alpha^N
\bigg|
\\&\label{rest:to-dir}
\leq
\frac{\norm{\varphi(s,\cdot)}_{\infty}}{4NB}
\int_{\Omega_N}
\sum_{x\in\T_N}
\sum_{j=0}^{\ell_N}
\sum_{i=\min\{0,\underline{r}_{N,j}\}}^{\max\{0,\underline{r}_{N,j}\}}
    \tau_{x+i}\mathbf{c}_{N}^{j}
    \big|
    \nabla_{x+i+1,x+i}[\sqrt{f}]
    \big|^2
\rmd\nu_\alpha^N
\\&\label{rest:to-zero}
+
\frac{B\norm{\varphi(s,\cdot)}_{\infty}}{2N}
\int_{\Omega_N}
\sum_{x\in\T_N}
\sum_{j=0}^{\ell_N}
\sum_{i=\min\{0,\underline{r}_{N,j}\}}^{\max\{0,\underline{r}_{N,j}\}}
    \tau_{x+i}\mathbf{c}_{N}^{j}
    \big(
    \theta_{x+i+1,x+i}+\mathbf{1}
    \big)[f]
\rmd\nu_\alpha^N
,
\end{align}
for an arbitrary $B>0$. Let us focus on the term \eqref{rest:to-zero}. Because $\tau_{x+i}\mathbf{c}_{N}^{j}$ is independent of the occupation at the sites $x+i$ and $x+i+1$, and $\alpha$ is a constant profile, the term \eqref{rest:to-zero} is bounded from above by
\begin{align}
\frac{B\norm{\varphi(s,\cdot)}_{\infty}}{N}
\int_{\Omega_N}
\sum_{x\in\T_N}
\sum_{j=0}^{\ell_N}
\sum_{i\in \mathcal{R}_{N,j}}
\tau_{x-i}\mathbf{c}_{N}^{j}(\mathbf{e}_{x-i,x-i+1}+\mathbf{e}_{x-i+1,x-i})
    \;f
\rmd\nu_\alpha^N
\\\leq 
\norm{\varphi(s,\cdot)}_{\infty}
B \underline{r}_N,
\end{align}
with $\underline{r}_N$ as in \eqref{total-rate}.

We now focus on \eqref{rest:to-dir} that is bounded from above by
\begin{align}
\frac{\norm{\varphi(s,\cdot)}_{\infty}}{4NB}
\sum_{x\in\T_N}
\int_{\Omega_N}
\sum_{i\in\T_N}
    \sum_{j=1}^{\ell_N}
    \tau_{x-i}\mathbf{r}_{N}^{m,j}
    \big|
    \nabla_{x+1,x}\sqrt{f}
    \big|^2
\rmd\nu_\alpha^N
\\
\leq 
\frac{\norm{\varphi(s,\cdot)}_{\infty}}{2BN^2}
\mathfrak{D}_N(\sqrt{f}|\nu_\alpha^N)
.
\end{align}

In this way, \eqref{rest:var} can be bounded from above by
\begin{align}
\frac{c_{\alpha}}{A}
+
BM_{T,\varphi}
\underline{r}_N
+\sup_{f \text{ density}}
\bigg\{
    \frac{1}{2N}\bigg(
    \frac{M_{t,\varphi}}{B N}
    -\frac{t}{A}
    \bigg)
    \mathfrak{D}_{N}(\sqrt{f}|\nu_\alpha^N)
\bigg\}
,
\end{align}
with $M_{t,\varphi}=\int_0^t\norm{\varphi(s,\cdot)}_{\infty}\rmd s$. Solving  
\begin{align}
    \begin{cases}
    \frac{t}{A}
    =\frac{M_{t,\varphi}}{B N}
    \\
    \frac{c_{\alpha}}{A}
    =BM_{t,\varphi}
\underline{r}_N
    \end{cases}
\end{align}
for $A$ and $B$ yields the upper bound in the statement of this lemma.

\end{proof}

\begin{Lemma}[One-block estimate]\label{lem:1block}
Fix $L\in\mathbb{N}_+$, $w\in\T_N$ and $y\notin B_{L}(w)$. For each $x\in\T_N$, let $\varphi_{x}:[0,T]\times \mathcal{D}_{\Omega_N}[0,T]\to\mathbb{R}$ be independent of the occupation-values in $\{x+w+r,x+y\}_{r\in B_{L}}$, and such that $\text{c}_{t,\varphi}:=\int_0^t\sup_{x\in\T_N,\eta\in\Omega_N}|\varphi_x(\cdot,s)|\rmd s<\infty $.

For every $t\in [0,T]$ it holds that
    \begin{align}\label{bound:lem1block}
        \mathbb{E}_{\mu_N}
        \bigg[
        \bigg|
        \int_{0}^t
        \frac{1}{N}\sum_{x\in\T_N}
        \varphi_{x}(s,\eta_{N^2s})
        \big(
        \inner{\tau_{x+w}\eta_{N^2s}}_{L}
        -\eta_{N^2s}(x+y)
        \big)
        \rmd s
        \bigg|
        \bigg]
        \leq 
        \sqrt{2c_\alpha}c_{t,\varphi}\frac{L+\rmd(w,y)|}{\sqrt{\mathfrak{p}_NN^2t}}
        .
    \end{align}
\end{Lemma}
\begin{proof}
From Proposition \ref{prop:base}, it is enough to estimate
    \begin{multline}\label{1block:var}
        \frac{c_{\alpha}}{A}
        +
        \int_0^t
        \sup_{f}
        \bigg\{
        \int_{\Omega_N}
        \frac{1}{N}\sum_{x\in\T_N}
        \varphi_{x}(s,\eta)
        \big(
        \inner{\tau_{x+w}\eta}_{L}
        -\eta(x+y)
        \big)
        f(\eta)\rmd\nu_\alpha^N(\eta)
        -\frac{\mathfrak{D}_{N}(\sqrt{f}|\nu_\alpha^N)}{2N A}
        \bigg\}
        \rmd s
    \end{multline}
with the $\sup$ taken over the set of densities with respect to $\nu_\alpha^N$. Expressing 
    \begin{align}
        \inner{\tau_{x+w}\eta}_{L}
        -\eta(x+y)
        =
        \frac{1}{L}
        \sum_{r\in B_{L}}
        \big\{
        \eta(x+w+r)-\eta(x+y)
        \big\},
    \end{align}
    from \eqref{by-parts}, we can bound from above the integral over $\Omega_N$ in \eqref{1block:var} by 
    \begin{align}\label{1block:0}
        \frac{\text{c}_{s,\varphi}'}{2L N}
        \sum_{x\in\T_N}
        \sum_{r\in B_{L}}
        \int_{\Omega_N}
       	(\mathbf{e}_{x+w+r,x+y}
        +\mathbf{e}_{x+y,x+w+r})
        |
        \nabla_{x+w+r,x+y}[f]
        |
        \rmd\nu_\alpha^N,
    \end{align}
where $\text{c}_{s,\varphi}':=\sup_{x\in\T_N}|\varphi_x(\cdot,s)|_{\infty}$.
    
Let us decompose
    \begin{align}
        \nabla_{x+w+r,x+y}[f](\eta)
        =\sum_{n=0}^{S-1}
        \nabla_{x_{n},x_{n+1}}[f](P_n\eta)
        , 
    \end{align}
with $S=2\rmd(w+r,y)-1$ and where, denoting by $\mathbf{1}_{\Omega_N}$ the identity operator in $\Omega_N$,
\begin{align}
            \begin{cases}
            P_{n+1}
            =\theta_{x_n,x_{n+1}}P_n,
            &0\leq n\leq S-1,\\
            P_0:=\mathbf{1}_{\Omega_N}
            ,&
        \end{cases}
\end{align}
and the ordered sequence of bonds $(\{x_n,x_{n+1}\})_n$ is defined through path exchanging the occupation-value at the sites $x+w+r$ and $x+y$. Precisely, letting (say) $x_0=x+y$, then $x_{S/2}=x+w+r$, $x_S=x_0$ and $\rmd(x_n,x_{n+1})=1$ for each $n$.

From the inequality $|a-b|\leq \frac{1}{2A}(\sqrt{a}-\sqrt{b})^2+A(a+b)$, for any $A>0$ and $a,b>0$, the quantity in \eqref{1block:0} can be bounded from above by
    \begin{multline}
        2\text{c}_{t,\varphi}\rmd(w+r,y)A
        +
        \frac{\text{c}_{s,\varphi}'}{2L NA}
        \sum_{r\in B_{L}}
        \sum_{n=0}^{2(L+\rmd(w,y))-1}
        \int_{\Omega_N}
        \sum_{x\in\T_N}
        \big|
        \nabla_{x_{n},x_{n+1}}\sqrt{f}(P_n\eta)
        \big|^2
        \rmd\nu_\alpha^N
        \\
        \leq
        2\text{c}_{s,\varphi}'(L +\rmd(w,y))A
        +
        \text{c}_{s,\varphi}'\frac{L+\rmd(w,y)}{ANN^2\mathfrak{p}_N}
        \mathfrak{D}_N^{}(\sqrt{f}|\nu_\alpha^N),
    \end{multline}

In this way, \eqref{1block:var} can be estimated from above by
    \begin{align}
        \frac{c_{\alpha}}{B}
        +
        2\text{c}_{t,\varphi}(L+\rmd(w,y))A
        +
        \frac{1}{N}\bigg(
        \frac{\text{c}_{t,\varphi}(L+\rmd(w,y))}{AN^2\mathfrak{p}_N}
        -\frac{t}{2B}
        \bigg)
        \mathfrak{D}_N(\sqrt{f}|\nu_\alpha^N)
        .
    \end{align}
Solving for $A$ and $B$ the system 
\begin{align}
    \begin{cases}
        \frac{t}{2B}
        =\frac{\text{c}_{t,\varphi}(L+\rmd(w,y))}{AN^2\mathfrak{p}_N}
        ,\\
        \frac{c_{\alpha}}{B}
        =
        2\text{c}_{t,\varphi}(L+\rmd(w,y))A
    \end{cases}
\end{align}
concludes the proof.    

\end{proof}

\begin{Lemma}[Two-block estimate]\label{lem:2block}
Fix $\ell,L\in\mathbb{N}_+$ such that $\ell>L$, and $y,w\in\T_N$ such that $B_{\ell}(w)\cap B_{L}(y)=\emptyset$ and $y<w$. For each $x\in\T_N$, let $\varphi_{x}:[0,T]\times \mathcal{D}_{\Omega_N}[0,T]$ be independent of the occupation-values in $B_{\ell}(x+w)\cup B_L(x+y)$ and such that $c_{t,\varphi}:=\int_{0}^t|\varphi_x(s,\cdot)|_{\infty}\rmd s<\infty$.

For every $t\in [0,T]$ it holds that
    \begin{multline}
        \mathbb{E}_{\mu_N}
        \bigg[
        \bigg|
        \int_{0}^t
        \frac{1}{N}\sum_{x\in\T_N}
        \varphi_{x}(s,\eta_{N^2s})
        \big(
        \inner{\tau_{x+w}\eta_{N^2s}}_{\ell}-\inner{\tau_{x+y}\eta_{N^2s}}_{L}
        \big)
        \rmd s
        \bigg|
        \bigg]
   \\
   \leq 
    \frac{\text{c}_{t,\varphi}k^{\star}}{L}
    +
    2\sqrt{2c_\alpha\text{c}_{t,\varphi}}
    \kappa^\star\sqrt{
    \frac{L^2}{\mathfrak{p}_N N^2t}
    +
    \frac{(\rmd(w+\ell,y))^2}{r_\star N^2t}
    }
    .
    \end{multline}
\end{Lemma}
\begin{proof}
From Proposition \ref{prop:base}, it is enough to estimate
    \begin{multline}\label{2block:var}
        \frac{c_{\alpha}}{A}
        +
        \int_0^t
        \sup_{f \text{ density}}
        \bigg\{
        \bigg|
        \int_{\Omega_N}
        \frac{1}{N}\sum_{x\in\T_N}
        \varphi_{x}(s,\eta)
        \big(
        \inner{\tau_{x+w}\eta}_{\ell}-\inner{\tau_{x+y}\eta}_{L}
        \big)
        f(\eta)\rmd\nu_\alpha^N(\eta)
        \bigg|
        \\-\frac{1}{2NA}
        \mathfrak{D}_{N}(\sqrt{f}|\nu_\alpha^N)
        \bigg\}
        \rmd s
        .
    \end{multline}

Let $\ell$ be divisible by $L$. Otherwise, one can replace $\ell/L$ by $\floor{\ell/L}$ for the rest of the proof, obtaining in the next display an additive error of the order of $1/L$. For any set $A\subset \T_N$, and $r\neq0$, we shall write $\ndef{rA}=\{ra:\;a\in A\}$. Note that $\cup_{p\in L B_{\ell/L}}B_{L}(p)$ forms a covering of $B_{\ell}$ by non-intersecting sets, and clearly, $|L B_{\ell/L}|=\ell/L$. 
Let us then decompose
\begin{align}
    \inner{\tau_{x+w}\eta}_{\ell}-\inner{\tau_{x+y}\eta}_{L}
    =
    \frac{1}{\ell}
        \sum_{\substack{p\in L B_{\ell/L}\\r\in B_{L}}}
    \big\{
    \eta(x+r+p+w)-\eta(x+r+y)
    \big\}.
\end{align}
From \eqref{by-parts} and the above, and shortening \ndef{$f_x\equiv\tau_{-x}f$}, the absolute value in the first line of \eqref{2block:var} equals 
\begin{align}
\frac{1}{N\ell}\sum_{x\in\T_N}
\int_{\Omega_N}
\varphi_x(s,\eta)
\sum_{\substack{p\in L B_{\ell/L}\\r\in B_{L}}}
    \big\{
    \eta(x+r+y)-\eta(x+r+p+w)
    \big\}
    \nabla_{r+p+w,r+y}[f_x]
    \rmd\nu_{\alpha}^N
\\\leq
\frac{|\varphi_x(s,\cdot)|_{\infty}}{N\ell}
\sum_{x\in\T_N}
\sum_{p\in L B_{\ell/L}}
    \int_{\Omega_N}
\sum_{r\in B_{L}}    
    (\mathbf{e}_{r+p+w,r+y}+\mathbf{e}_{r+y,r+p+w})
    \big|
    \nabla_{r+p+w,r+y}[f_x]
    \big|
    \rmd\nu_{\alpha}^N,
    \label{2block:base}
\end{align}
where we used the translation invariance of $\nu_\alpha^N$.

We are going to reduce our state space to configurations where is possible to construct a mobile cluster. For each $p,w,y$ as in the previous display, let us consider the auxiliary space
\begin{align}
    \begin{split}\label{2block:right-dens}
    \Omega_{p,w,y}
    =
    \bigg\{
    \eta\in\Omega_N\mid 
    \;
    &\; 
    \big(
    \inner{\tau_{p+w}\eta}_{L}
    \geq \frac{k^{\star}+1}{L}
    \;\;\text{or}\;\; 
    \inner{\tau_{y}\eta}_{L}
    \geq \frac{k^{\star}+1}{L}
    \big)
    \\
    \text{and}\;&\; 
    \big(
    \inner{\tau_{p+w}\overline{\eta}}_{L}
    \geq \frac{k^{\star}+1}{L}
    \;\;\text{or}\;\; 
    \inner{\tau_{y}\overline{\eta}}_{L}
    \geq \frac{k^{\star}+1}{L}
    \big)
    \bigg\}.
    \end{split}
\end{align}
It is important to note that if $\eta\notin \Omega_{p,w,y}$, then there are at most $k^{\star}+1$ values of $r\in B_{L}$ such that $\mathbf{e}_{r+p+w,r+y}(\eta),\mathbf{e}_{r+y,r+p+w}(\eta)\neq 0$. As such, the quantity in \eqref{2block:base} restricted to $(\Omega_{p,w,y})^c$ is no larger than $2|\varphi_x(s,\cdot)|_{\infty}(k^{\star}+1)/L$.

If $\eta\in\Omega_{p,w,y}$ then either:
\begin{enumerate}
    \item At least one of the boxes $B_L(p+w)$ and $B_L(y)$ contains at least $k^{\star}+1$ particles and at least $k^{\star}+1$ vacant sites;
    \item The box $B_L(p+w)$ (resp. $B_L(y)$) contains at least $k^{\star}+1$ particles and the box $B_L(y)$ (resp. $B_L(p+w)$) contains at least $k^{\star}+1$ vacant sites.
\end{enumerate}
Let us focus on (2), and suppose that $B_L(p+w)$ has the target number of particles and $B_L(y)$ of vacant sites. Note that if some box contains \textit{less} than $k^{\star}+1<L/2$ particles, then it contains \textit{more} than $k^{\star}+1$ vacant sites. In particular, there is a pair of boxes $B_L(x_0),B_L(x_0+1)$, contained in the segment connecting $p+w$ and $y$, such that $B_L(x_0)$ contains at least $k^{\star}+1$ particles and $B_L(x_0+1)$ at least $k^{\star}+1$ vacant sites. The remaining case in (2) is analogous, providing $x_0$ such that the box centred at $x_0$ contains at least $k^{\star}+1$ vacant sites, while the one at $x_0+1$ at least $k^{\star}+1$ particles.

From the previous observations, provided $\eta\in\Omega_{p,w,y}$, let $\ndef{x^\star\equiv x^\star(\eta)}\neq r+p+w,r+y$ be a fixed site in a box of length $2L$ contained in the segment connecting $p+w$ and $y+L$, and such that $B_{2L}(x^\star)$ contains, at least, $k^{\star}+1$ particles and vacant sites.

We will estimate 
\begin{align}\label{2block:cutoff}
\frac{|\varphi_x(s,\cdot)|_{\infty}}{N\ell}
\sum_{x\in\T_N}
        \sum_{\substack{p\in L B_{\ell/L}\\r\in B_{L}}}
    \int_{\Omega_{p,w,y}}
    \mathbf{e}_{r+p+w,r+y}
    \big|
    \nabla_{r+p+w,r+y}[f_x]
    \big|
    \rmd\nu_{\alpha}^N
    .
\end{align}
The term associated with $\mathbf{e}_{r+y,r+p+w}$ is completely analogous. Consider the auxiliary set $\ndef{\Omega_\star}\subset\Omega_{p,w,y}$ composed by the configurations such that there exist a box of length $L$ in the segment connecting $p+w$ and $y$ containing a \textit{mobile cluster}.

For each $\eta\in\Omega_{p,w,y}$ let us split
\begin{align}\label{2block:cons-mc}
    \big|
    \nabla_{r+p+w,r+y}[f_x](\eta)
    \big|
    &\leq 
    \big|
    f_x(\eta^\star)
    -f_x(\eta)
    \big|
    +
    \big|
    \theta_{r+p+w,r+y}f_x(\eta)-\theta_{r+p+w,r+y}f_x(\eta^\star)
    \big|
    \\&+
    \big|
    \nabla_{r+p+w,r+y}[f_x](\eta^\star)
    \big|
    ,
\end{align}
where $\eta^\star\in\Omega_\star$ is the result of the following procedure. Let us focus on the first term in the right-hand side above. Fix some $\ndef{y^\star}\in B_{2L}(x^\star)$ such that $r+p+w,r+y\notin B_{}(y^\star)$, and shorten $\ndef{B^\star}\equiv B_{k^{\star}}(y^\star)$. Either $B^\star$ is a mobile cluster, or have too many particles, or to few.
 \begin{enumerate}
    \item If $B^\star$ has too many particles, each of the excess of particles are moved, through nearest-neighbour jumps, to their closest vacant site in $B_{2L}(x^\star)\backslash B^\star$;
    
    \item If $M$ has too many vacant sites, the required number of particles to form a mobile cluster, that are in $B_{2L}(x^\star)\backslash B^\star$ and the closest to $B^\star$, are moved to inside of $B^\star$.
\end{enumerate}
Note that because in \eqref{2block:right-dens} we require at least one more particle or vacant site than the length of the mobile cluster, then all mixing of the system previously described can be performed leaving the occupation-value at $r+p+w$ and $r+y$ invariant in the end, as they do not need to be used in the mobile clusters. 

This mass transportation can be encapsulated into an ordered sequence of nodes characterizing the nearest-neighbour exchanges, with a total of $S(\eta)\leq k^{\star}2L$ steps, that we write as $\{x_n(\eta),x_{n+1}(\eta)\}_{n=0,\dots, S(\eta)-1}$. Let us also associate to it the sequence of transformations defined through $P_{n+1}=\theta_{x_n,x_{n+1}}P_{n}$, for each $0\leq n\leq S-1$ and with $P_0=\mathbf{1}_{\Omega_N}$ the identity operator in $\Omega_N$.  From Young's inequality, for any $A_0>0$,
\begin{align}\label{2block:ref0}
    \big|
    f_x(\eta^\star)
    -
    f_x(\eta)
    \big|
    \leq 
    A_0
    \sum_{n=0}^{S-1}
    \big(
    P_{n+1}f_x+P_{n}f_x
    \big)
    +
    \frac{1}{2A_0}
    \sum_{n=0}^{S-1}
    \big|
    \nabla_{x_{n},x_{n+1}}[\sqrt{P_{n}f_x}]
    \big|^{2}.
\end{align}
Therefore,
\begin{multline}\label{2block:ref1}
    \frac{1}{N\ell}
    \sum_{x\in\T_N}
        \sum_{\substack{p\in L B_{\ell/L}\\r\in B_{L}}}
    \int_{\Omega_{p,w,y}}
    \mathbf{e}_{r+p+w,r+y}
    \big|
    f_x(\eta^\star)
    -
    f_x(\eta)
    \big|
    \rmd\nu_{\alpha}^N
    \\
    \leq 
    4A_0 k^{\star}L
    +
    \frac{1}{2A_0N\ell}
    \sum_{x\in\T_N}
        \sum_{\substack{p\in L B_{\ell/L}\\r\in B_{L}}}
    \int_{\Omega_{p,w,y}}
    \sum_{n=0}^{S-1}
    \big|
    \nabla_{x_{n},x_{n+1}}[\sqrt{P_{n}f_x]}
    \big|^{2}
    \rmd\nu_{\alpha}^N
\end{multline}
where we bounded from above $S(\eta)<2L k^{\star}$, and used the invariance of $\nu_\alpha^N$ for the exchange of occupation values.

The treatment of the term resulting from the first quantity in the right-hand side of \eqref{2block:cons-mc} is concluded by bounding from above 
\begin{align}\label{2block:ref2}
    \sum_{x\in\T_N}
    \int_{\Omega_{p,w,y}}
    \sum_{n=0}^{S-1}
    \big|
    \nabla_{x_{n},x_{n+1}}[\sqrt{P_{n}f_x]}
    \big|^{2}
    \rmd\nu_{\alpha}^N
    \leq 
    \frac{4k^{\star}L}{N^2\mathfrak{p}_N}
    \mathfrak{D}_N(\sqrt{f}\mid\nu_\alpha^N)  
    .
\end{align}
The second quantity in \eqref{2block:cons-mc} is estimated analogously from the invariance of $\nu_\alpha^N$ for $\theta_{r+p+w,r+y}$.

Now for the term arising from \eqref{2block:cutoff} and associated with the third term in the right-hand side of \eqref{2block:cons-mc}, we need to estimate
\begin{align}\label{2block:mc-to-flip0}
    \frac{1}{N\ell}
    \sum_{\substack{p\in L B_{\ell/L}\\r\in B_{L}}}
    \sum_{x\in\T_N}
    \int_{\Omega_{p,w,y}}
    \mathbf{e}_{r+p+w,r+y}(\eta^\star)
    \big|
    \nabla_{r+p+w,r+y}[f_x](\eta^\star)
    \big|
    \rmd\nu_{\alpha}^N(\eta^\star),
\end{align}
where $\eta^\star=P_{S}\eta$. In order to exchange the occupation-values in $\{r+p+w,r+y\}$, fixed $\eta^\star\in \Omega_{\star}$, we consider the following path:
\begin{enumerate}
    \item the mobile cluster in $B^\star$ is moved to the vicinity of the site $r+p+w$ through a sequence of no more than $k^{\star}\rmd (y^\star,r+p+w)<k^{\star}\rmd(w+\ell,y)$  nearest-neighbour jumps; 
    
    \item the particle located at $r+p+w$ is moved, with the assistance of the mobile cluster, to a vicinity of the site $r+y$ with no more than $k^{\star}\rmd (p+w,y)<k^{\star}\rmd(w+\ell,y)$ steps; 
    
    \item the vacant site at $r+y$ is exchanged with the particle that was originally at $r+p+w$; 
    
    \item the mobile cluster transports to the site $r+p+w$ the vacant site at the vicinity of $r+y$, that was originally at $r+y$, with no more than $k^{\star}\rmd(p+w,y)<k^{\star}\rmd(w+\ell,y)$ steps; 
    \item the mobile cluster is transported back to its original position with no more than $k^{\star}\rmd(w+\ell,y)$ steps. 
\end{enumerate}

Let $\{y_n,y_{n+1}\}_{n=0,\dots,S^\star-1}$ be the sequence of nodes where the occupation variables are exchanged according to the procedure just described, and $Q_{n+1}=\theta_{y_n,y_{n+1}}Q_{n}$ with $Q_0=\mathbf{1}_{\Omega_N}$. It is crucial to note that for $\eta^\star\in\Omega_\star$ it holds $\mathbf{r}^{\star}(Q_{n}\eta^\star)\geq r_0$. One may now proceed analogously to \eqref{2block:ref0},\eqref{2block:ref1} and \eqref{2block:ref2} to see that \eqref{2block:mc-to-flip0} is no larger than 
\begin{align}
    4k^{\star}\rmd(w+\ell,y)A_{k^{\star}} 
    +
    \frac{1}{2r_0A_{k^{\star}}N\ell}
    \sum_{x\in\T_N}
    \sum_{\substack{p\in L B_{\ell/L}\\r\in B_{L}}}
    \int_{\Omega_{\star}}
    \sum_{n=0}^{S^\star-1}
    Q_{n}\mathbf{r}^{\star}
    \big|
    \nabla_{y_{n},y_{n+1}}[\sqrt{Q_{n}f_x}]
    \big|^{2}
    \rmd\nu_{\alpha}^N
    \\
    \leq 
    4k^{\star}\rmd(w+\ell,y) 
    +
    \frac{4k^{\star}\rmd(w+\ell,y)}{r_0A_{k^{\star}} NN^2}
    \mathfrak{D}_N(\sqrt{f}\mid\nu_\alpha^N)
    ,
\end{align}
for any $A_{k^{\star}}>0$.

Collecting the inequality in the previous display, \eqref{2block:ref1} and \eqref{2block:ref2}, then applying them to \eqref{2block:cutoff} with the triangle inequality and \eqref{2block:cons-mc}, we conclude the treatment of \eqref{2block:cutoff}. This implies that \eqref{2block:var} is bounded from above by
\begin{multline}
    \frac{c_{\alpha}}{A}
    +\text{c}_{t,\varphi}\frac{k^{\star}}{L}
    +4k^{\star}LA_0
    +4k^{\star}A_{k^{\star}}\rmd(w+\ell,y) 
    \\+
    \frac1N\bigg(
    \frac{4\text{c}_{t,\varphi}}{r_0}\frac{k^{\star}\rmd(w+\ell,y)}{A_{k^{\star}} N^2}
    +2\text{c}_{t,\varphi}\frac{k^{\star}L}{A_0N^2\mathfrak{p}_N}
    -\frac{t}{2A}
    \bigg)
    \mathfrak{D}_N(\sqrt{f}\mid\nu_\alpha^N)
    .
\end{multline}
Solving for $A_0,A_{k^{\star}}$ and $A$ the system
\begin{align}
    \begin{cases}
    \frac{4\text{c}_{t,\varphi}}{r_0}\frac{k^{\star}\rmd(w+\ell,y)}{A_{k^{\star}} N^2}
    =\frac{t}{4A}
    \\
    2\text{c}_{t,\varphi}\frac{k^{\star}L}{A_0N^2\mathfrak{p}_N}
    =\frac{t}{4A}
    \\
    \frac{c_{\alpha}}{A}
    =4k^{\star}LA_0
	+4k^{\star}A_{k^{\star}}\rmd(w+\ell,y) 
    \end{cases}
\end{align}
gives the upper-bound in the statement.
	
\end{proof}

\section{Energy Estimate}\label{sec:energy}
The goal of this section is to prove that $ \Phi(\rho)\in L^2([0,T];\mathcal{H}^1(\mathbb{T})) $. From \cite[Lemma A.1.9]{phd:adriana}, it is enough to show that there exists a function $ \partial\Phi\in L^2([0,T]\times \mathbb{T}) $ satisfying for all $ G\in C^{0,1}([0,T]\times\mathbb{T}) $ the identity
\begin{align*}
	\int_0^T\inner{\partial_uG,\Phi}
	\rmd s
	=-
	\int_0^T\inner{G,\partial\Phi}
	\rmd s
	.
\end{align*}
As explained in \cite[Section 5]{s2f}, to see this it is enough to show that an energy functional on $\mathcal{D}([0,T],\mathcal{M}_+)\to\mathbb{R}\cup\{\infty\} $ associated with $\Phi$ is bounded, then argue via Riez's representation theorem. Regarding the latter, that is what we show here, there should exist finite constants $ \kappa,K>0 $ such that
\begin{align}\label{energy:bound}
	\text{E}_{\mathbb{Q}}
	\left[
	\sup_{G\in C^{0,1}([0,T]\times\mathbb{T})}
	\int_0^T
	\bigg\{
	\inner{\partial_u G,\Phi(\rho)}
	-
	\kappa\norm{G}_{L^2(\T)}^2
	\bigg\}
	\rmd s
	\right]
	\leq K
	.
\end{align}
We are going to show that the above holds with $\kappa=|\mathbf{c}_{\ell_N}|_{\infty}$ and $K=\text{c}_\alpha$, where we recall $\text{c}_\alpha$ from the proof of Proposition \ref{prop:base} to be such that $\text{H}(\mu_N|\nu_\alpha^N)\leq Nc_{\alpha}$.

From the uniform convergence of $(\Phi_L)_L$ towards $\Phi$, it is enough to study the functional given by replacing in the previous display $\Phi$ by $\Phi_L$, for a fixed $L$ that shall be taken to $+\infty$ in the end.

It is a standard procedure to consider a countable dense subset $ \{G^{[p]}\}_{p\in\mathbb{N}} $ of $ C^{0,1}\left([0,T]\times \mathbb{T}\right) $, with respect to the norm $\int_0^T\norm{\cdot}_{L^2(\T)}\rmd s$, reducing the $\sup$ to a limit and a $\max$ of a finite collection (see \cite[Proof of Proposition 5.5 up to equation (72)]{s2f}); then argue via the already proved one and two-blocks estimates (Lemmas \ref{lem:1block} and \ref{lem:2block}). This implies that it is enough to prove that

\begin{align}\label{fde:eq:energy-h}
	\limsup_{\eps\to 0}\lim_{\ell\to+\infty}
	\liminf_{N\to+\infty}
	\mathbb{E}_{\mu_N}
	\Bigg[
	\max_{{p\leq \ell}}
	\int_0^T
	\Bigg\{
	\frac{1}{N}
	\sum_{x\in\mathbb{T}_N}
	\mathbf{h}_L(\tau_x\eta_{N^2s})
	\partial_u G^{[p]}(s,\tfrac{x}{N})
	-\norm{G^{[p]}}_{L^2}^2
	\Bigg\}
	\rmd s
	\Bigg]
	\leq K.
\end{align} 
The $\max$ above can be replaced by a summation over $p\leq \ell$, and summing and subtracting the appropriate terms inside the summation over $x\in\T_N$ above, $\mathbf{h}_L$ is replaced by $\mathbf{h}_L+\mathbf{g}_{\ell_N}$ by invoking Lemma \ref{lem:rest}. From Assumption \ref{ass}, $\mathbf{h}_L$ can be further replaced by $\mathbf{h}_{\ell_N}$. We recall that $\mathbf{H}_{\ell_N}=\mathbf{h}_{\ell_N}+\mathbf{g}_{\ell_N}$. Proceeding then with similar arguments as in Proposition \ref{prop:base}, one can reduce all of this to the need of showing that 
\begin{align}\label{energy:var}
\text{c}_{\alpha}
+
\limsup_{\eps\to 0}
\liminf_{N\to+\infty}
\int_0^T\sup_{f}
	\bigg\{
\int_{\Omega_N}	
f\frac{1}{N}
	\sum_{x\in\mathbb{T}_N}
	\tau_x\mathbf{H}_L
	\partial_u G_s(\tfrac{x}{N})
	\rmd\nu_\alpha^N
	-N\frac12\mathfrak{D}_N(\sqrt{f}|\nu_\alpha^N)
\\	-\kappa\norm{G_s}_{L^2}^2
	\bigg\}
	\rmd s
	\leq K,
\end{align}
where the supremum is over densities with respect to $ \nu_\alpha^N $, and $G\in C^{0,1}([0,T]\times\mathbb{T})$ is arbitrary. In order to prove the previous affirmations, the reader can follow the exposition in \cite[From equation (5.3) to (5.4)]{s2f}.

Performing a Taylor expansion, the space derivative can be replaced by its discrete version $N\nabla^+$, with a cost which vanishes as $N\to+\infty$ since $|\mathbf{H}_{\ell_N}|_{\infty}<\infty$ by Assumption \ref{ass}. This discrete derivative is then passed to $\mathbf{H}_{\ell_N}$ performing a summation by parts, allowing us to invoke the gradient property and then bounding from above the integral over $\Omega_N$ in the previous display by
\begin{align}
	\frac{1}{N}
	\sum_{x\in\mathbb{T}_N}
	G_s(\tfrac{x}{N})
	\int_{\Omega_N}	
	\mathbf{c}_{\ell_N}(\tau_x\eta)
	(\eta(x+1)-\eta(x))
	f(\eta)
	\rmd\nu_\alpha^N(\eta)
	+|\mathbf{H}_{\ell_N}|_{\infty}\text{e}(N)
	,
	&&\lim_{N\rightarrow+\infty}|\text{e}(N)|=0.
\end{align}
From observation \eqref{by-parts} the absolute value of the first quantity in the previous display equals 
\begin{align}
	\bigg|
	\frac{1}{N}
	\sum_{x\in\mathbb{T}_N}
	\int_{\Omega_N}	
	\mathbf{r}_{\ell_N}(\tau_x\eta)
	(\eta(x+1)-\eta(x))
	G_s(\tfrac{x}{N})
	\nabla_{x,x+1}[f](\eta)
	\rmd\nu_\alpha^N(\eta)
	\bigg|
	\\\label{fde:eq:energy}
	\leq 
	\frac{1}{2A}\sum_{x\in\mathbb{T}_N}\sum_{\eta\in \Omega_N}
	\mathbf{r}_{\ell_N}(\tau_x\eta)
	\left(G_s(\tfrac{x+1}{N})\right)^2
	\left(\sqrt{f}(\eta)+\sqrt{f}(\eta^{x,x+1})\right)^2
	\nu_{\gamma}^N(\eta)
	+\frac{A}{2N^2}\mathfrak{D}_N(\sqrt{f}|\nu_\alpha^N)
		\\
	\leq
	\frac{|\mathbf{c}_{\ell_N}|_{\infty}}{A}
	\sum_{x\in\mathbb{T}_N}\left(G_s(\tfrac{x+1}{N})\right)^2
	+\frac{A}{2N^2}\mathfrak{D}_N(\sqrt{f}|\nu_\alpha^N)
	.
\end{align}

Recalling \eqref{energy:var}, in order to conclude the proof it is enough to fix $A=N$, $\kappa=|\mathbf{c}_{\ell_N}|_{\infty}$ and $K=\text{c}_\alpha$.

\begin{Rem}[Regime II]\label{rem:fde-energy}
	In the Regime \ref{II} one can show instead that $ \rho\in L^2([0,T];\mathcal{H}^1(\mathbb{T})) $. This corresponds to showing \eqref{energy:bound} with $\Phi(\rho)$ replaced by $\rho$, and the proof is much simpler since there is no truncation step needed and the function $\Phi(\rho)\equiv\rho$ is linear. As such, there is no need to introduce $\mathbf{g}_{\ell_N}$ and we have simply \eqref{fde:eq:energy-h} with $\tau_x\mathbf{h}_{L}$ there replaced by $\pi_x$. This leads instead to \eqref{fde:eq:energy} with $\mathbf{r}_{\ell_N}$ replaced by $\mathbf{e}_{0,1}+\mathbf{e}_{1,0}$, and $\mathfrak{D}_N$ replaced by $\tfrac{1}{\mathfrak{m}}\mathfrak{D}_N$. Consequently, $A=N/\mathfrak{m},\kappa=\mathfrak{m}$ and $K=\text{c}_\alpha$. 
\end{Rem}

\addtocontents{toc}{\protect\setcounter{tocdepth}{0}}

\subsection*{Funding:}
This project was partially supported by the ANR grant MICMOV (ANR-19-CE40-0012) of the French National Research Agency (ANR).

\addtocontents{toc}{\protect\setcounter{tocdepth}{2}}

\bibliographystyle{plain}
\bibliography{02.biblio}

\end{document}